\documentclass[12pt]{article}
\usepackage[latin1]{inputenc}
\usepackage{amsmath,amsthm,amssymb}
\usepackage{amsfonts}
\usepackage{amsmath,amsthm,amssymb,amscd}
\usepackage{latexsym}
\usepackage{color}
\usepackage{color,enumitem,graphicx}
\usepackage[colorlinks=true,urlcolor=blue,
citecolor=red,linkcolor=blue,linktocpage,pdfpagelabels,
bookmarksnumbered,bookmarksopen]{hyperref}
\usepackage{graphicx}
\usepackage{mathrsfs}
\usepackage{bm}
\usepackage{cite}

\makeatletter \@addtoreset{equation}{section} \makeatother

\newtheorem{theorem}{Theorem}[section]

\newtheorem{proposition}{Proposition}[section]
\newtheorem{lemma}{Lemma}[section]
\newtheorem{remark}{Remark}[section]

\numberwithin{equation}{section}

\allowdisplaybreaks

\begin{document}

\title{Normalized solutions for a fractional Choquard-type equation with exponential critical growth in $\mathbb{R}$}
\author
{Wenjing Chen\footnote{Corresponding author.}\ \footnote{E-mail address:\, {\tt wjchen@swu.edu.cn} (W. Chen),   {\tt sunqianva@163.com} (Q. Sun),  {\tt zxwangmath@163.com} (Z. Wang)}\ ,  Qian Sun  and   Zexi Wang\\
\footnotesize  School of Mathematics and Statistics, Southwest University,
Chongqing, 400715, P.R. China}
\date{ }
\maketitle

\begin{abstract}
{In this paper, we study the following fractional Choquard-type equation with prescribed mass
 \begin{align*}
   \begin{cases}
   (-\Delta)^{1/2}u=\lambda u +(I_\mu*F(u))f(u),\ \  \mbox{in}\ \mathbb{R},\\
   \displaystyle\int_{\mathbb{R}}|u|^2 \mathrm{d}x=a^2,
   \end{cases}
   \end{align*}
where  $(-\Delta)^{1/2}$ denotes the $1/2$-Laplacian operator, $a>0$, $\lambda\in \mathbb{R}$, $I_\mu(x)=\frac{{1}}{{|x|^\mu}}$ with $\mu\in(0,1)$, $F(u)$ is the primitive function of $f(u)$, and $f$ is a continuous function with exponential critical growth in the sense of the Trudinger-Moser inequality. By using a minimax principle based on the homotopy stable family, we obtain that there is at least one normalized ground state solution to the above equation.
}

\end{abstract}

\section{{\bfseries Introduction}}\label{introduction}

Recently, the following time-dependent fractional Choquard-type equation
\begin{align}\label{1.1}
i\frac{\partial\Psi}{\partial t}=(-\Delta)^s \Psi -(I_\mu*F(\Psi))f(\Psi), \quad \mbox{in} \   \mathbb{R} \times \mathbb{R}^N
\end{align}
has attracts much attention, where $i$ denotes the imaginary unit, $s\in (0,1)$,
 $I_\mu(x)=\frac{{1}}{{|x|^\mu}}$ with $\mu\in (0,N)$, $F$ is the primitive function of $f$, and $(-\Delta )^s$ is the fractional Laplacian operator defined by
\begin{align*}
(-\Delta )^su(x):=C({N,s})\ \mbox{P.V.}\int_{\mathbb{R}^N}\frac{u(x)-u(y)}{|x-y|^{N+2s}}\mathrm{d}y, \quad \text {in $\mathbb{R}^N$},
\end{align*}
for $u\in C_0^\infty(\mathbb{R}^N)$,
where P.V. means the Cauchy principal value and $C({N,s})$ is some positive normalization constant, we refer to \cite{di} for more details.
When we searching for stationary waves of $(\ref{1.1})$ with the form $\Psi(t,x)=e^{-i\lambda t}u(x)$ and $\lambda\in\mathbb{R}$, then $u$ solves the following equation
\begin{align}\label{1.2}
(-\Delta)^su=\lambda u +(I_\mu*F(u))f(u), \quad \text{in $\mathbb{R}^N$}.
\end{align}
To get solutions of $(\ref{1.2})$, one way is to fix $\lambda\in \mathbb{R}$ and look for solutions of $(\ref{1.2})$
as critical points of the energy functional $\hat J:H^{s}(\mathbb{R}^{N})\rightarrow \mathbb{R}$ (see e.g. \cite{cle,dSS,SGY,ZW})
\begin{align*}
\hat J(u)=\frac{1}{2} \int_{\mathbb{R}^N}(|(-\Delta)^{\frac{s}{2}}u|^2-\lambda u^2
)\mathrm{d}x-\frac{1}{2} \int_{\mathbb{R}^N}(I_\mu*F(u))F(u)\mathrm{d}x
\end{align*}
with
\begin{equation*}
	\int_{\mathbb{R}^N}|(-\Delta)^{\frac{s}{2}}u|^2dx=\int_{\mathbb{R}^N}\int_{\mathbb{R}^N}\frac{|u(x)-u(y)|^2}{|x-y|^{N+2s}}\,\mathrm{d}x\mathrm{d}y,
\end{equation*}
where $H^{s}(\mathbb{R}^{N})$ is a Hilbert space with the inner product and norm respectively
$$
\langle u,v\rangle=\int_{\mathbb{R}^{N}}(-\Delta)^{s/2}u(-\Delta)^{s/2}v\mathrm{d}x+\int_{\mathbb{R}^{N}} uv\mathrm{d}x,
$$
$$
\|u\|=\Big(\int_{\mathbb{R}^{N}}|(-\Delta)^{s/2}u|^{2}\mathrm{d}x+\int_{\mathbb{R}^{N}}|u|^{2}\mathrm{d}x\Big)^{1/2}.
$$

Another important way is to prescribe the $L^2$-norm of the unknown $u$, and $\lambda\in \mathbb{R}$ appears as a Lagrange multiplier, that is to consider the following problem
\begin{equation}\label{problem1}
	 \begin{cases}
   (-\Delta)^su=\lambda u +(I_\mu*F(u))f(u),\ \  \mbox{in}\ \mathbb{R}^N,\\
   \displaystyle\int_{\mathbb{R}^N}|u|^2 \mathrm{d}x=a^2.
   \end{cases}
\end{equation}
for any fixed $a>0$. This type of solutions is called normalized solution, and can be obtained by
 looking for critical points of the following energy functional
\begin{align*}
\bar J(u)=\frac{1}{2} \int_{\mathbb{R}^N} |(-\Delta)^{\frac{s}{2}}u|^2
\mathrm{d}x-\frac{1}{2} \int_{\mathbb{R}^N}(I_\mu*F(u))F(u)\mathrm{d}x
\end{align*}
on the $L^2$-sphere $$\bar S(a):=\Big\{u\in H^s(\mathbb{R}^N):\int_{\mathbb{R}^N}|u|^2 \mathrm{d}x=a^2\Big\}.$$
In particular, we are interested in looking for ground state solutions,  i.e., solutions minimizing $\bar J$ on $\bar S(a)$ among all nontrivial solutions, and the associated energy is called ground state energy.

The study of normalized solutions for the following semilinear elliptic equation
\begin{equation}\label{problem2}
	 \begin{cases}
   -\Delta u=\lambda u +g(u),\ \  \mbox{in}\ \mathbb{R}^N,\\
  \displaystyle \int_{\mathbb{R}^N}|u|^2 \mathrm{d}x=a^2,
   \end{cases}
\end{equation}
goes back to the pioneering work \cite{S1,S2} by Stuart. In that paper, Stuart dealt with problem \eqref{problem2} for $g(u)=|u|^{p-2}u$ and $p\in(2,2+\frac{4}{N})$ ($L^2$-subcritical case), here $2+\frac{4}{N}$ is called the $L^2$-critical exponent, which comes from the Gagliardo-Nirenberg inequality \cite{Nirenberg1}. When $g$ is $L^2$-supercritical growth, by using a mountain pass structure for the stretched functional, Jeanjean \cite{jeanjean} first obtained a normalized ground state solution of problem \eqref{problem2}.
By using a minimax principle based on the homotopy stable family,
Bartsch and Soave \cite{soave1,soave2} also presented a new approach that is based on a natural constraint associated to the problem and proved the existence of normalized solutions for problem \eqref{problem2}. Inspired by \cite{jeanjean,soave1,soave2}, Soave \cite{soave} studied problem \eqref{problem2} with combined nonlinearities $g(u)=\mu|u|^{q-2}u+|u|^{p-2}u$, $2<q\leq 2+\frac{4}{N}\leq p<2^*$ and $q<p$, where $2^*=\infty$ if $N\leq2$ and $2^*=\frac{2N}{N-2}$ if $N\geq3$. The Sobolev critical case $p=2^*$ and $N\geq3$ was considered by Soave \cite{soave3}. Furthermore, the above results are later generalized to exponential critical case by Alves et al.\cite{alvesji} and the fractional Laplace case by Luo and Zhang \cite{LZ}, Zhang and Han \cite{ZH}, Zhen and Zhang \cite{ZZ}.
More results of normalized solutions for scalar equations and systems can be found in \cite{barde,bartschjean1,bartschjean,BM,Li1,zhongzou,jeanjeanlu,wei}.

If $s=1$, Li and Ye \cite{liye}  considered problem $(\ref{problem1})$. Under a set of assumptions on $f$, with a similar argument of \cite{jeanjean}, they obtained a ground state solution when $N\geq3$. Yuan et al. \cite{yuan} generalized the above result from $f\in C^1(\mathbb{R},\mathbb{R})$ to  $f\in C(\mathbb{R},\mathbb{R})$ and obtained a ground state solution when $N=3$. Furthermore, Bartsch et al. \cite{bartsch} studied problem $(\ref{problem1})$ in all dimensions, and obtained infinitely many radial solutions if $f$ is odd.
In the case $N=2$ and $f$ has exponential critical growth, the existence of normalized solutions of \eqref{problem1} has been discussed by Deng and Yu \cite{dengyu}.
The fractional case of problem $(\ref{problem1})$ with general nonlinearities is also considered, see \cite{liluo} for more details.
For more results of normalized solutions for the Choquard equation, we refer the readers to \cite{cao,chenwang1,chenwang,liluo,liluoyang,yao} and references therein.
In particular,
Chen and Wang \cite{chenwang1} studied normalized solutions for a biharmonic equation with a Choquard nonlinearity involving exponential critical growth in $\mathbb{R}^4$. Different from the method of \cite{chenwang1}, the authors in \cite{chenwang} gave an another view to the same problem by using the minimax principle based on the homotopy stable family and used a more natural growth condition to estimate the upper bound of the ground state energy based on the Adams function \cite{LY3}.

Motivated by the results above, considering that there is no results on normalized solutions for the fractional Choquard-type equation with exponential critical growth in dimension one.
In this paper, we are interested in the following problem
\begin{equation}\label{problem}
	 \begin{cases}
   (-\Delta)^{1/2}u=\lambda u +(I_\mu*F(u))f(u),\ \  \mbox{in}\ \mathbb{R},\\
  \displaystyle \int_{\mathbb{R}}|u|^2 \mathrm{d}x=a^2,
   \end{cases}
\end{equation}
where  $(-\Delta)^{1/2}$ denotes the 1/2-Laplacian operator, $a>0$ is prescribed, $\lambda\in \mathbb{R}$, $I_\mu(x)=\frac{{1}}{{|x|^\mu}}$ with $\mu\in(0,1)$, $F(u)$ is the primitive function of $f(u)$, and $f$ is a continuous function with exponential critical growth. In order to use a variational approach, the maximal growth is motivated by
the Trudinger-Moser inequality first given by Ozawa \cite{ozawa} and later extended by Iula et al. \cite{lula} (see also \cite{kozono,Taka}). 
More precisely, we assume that $f$ satisfies:

$(f_1)$   $f\in C(\mathbb{R},\mathbb{R})$, and $\lim\limits_{t\to0}\frac{|f(t)|}{|t|^\kappa}=0$ for some $\kappa>2-\mu$;

$(f_2)$  $f$ has exponential critical growth at infinity, i.e.,
\begin{align*}
\lim\limits_{|t|\to+\infty}\frac{|f(t)|}{e^{\alpha t^2}}=
\begin{cases}
0,\quad  &\mbox{for}\ \ \alpha>\pi,\\
+\infty,\quad &\mbox{for}\ \ 0<\alpha <\pi;
\end{cases}
\end{align*}

$(f_3)$ There exists a constant $\theta>3-\mu$ such that $0<\theta F(t)\leq tf(t)$ for all $t\neq 0$; 

$(f_4)$ There exist $M_0>0$ and $t_0>0$ such that $F(t)\leq M_0|f(t)|$ for any $|t|\geq t_0$.

$(f_5)$ For any $t\in{\mathbb{R} \backslash \{0\}}$, define $\widetilde F(t):=f(t)t-(2-\mu)F(t)$, then
\begin{align*}
\frac{\widetilde F(t)}{|t|^{3-\mu}} \ \mbox{is non-increasing in $(-\infty,0)$ and non-decreasing in $(0,+\infty)$.}
\end{align*}

$(f_6)$ There exists $\beta_0>0$ such that $\liminf\limits_{t\to+\infty}\frac{f(t)}{e^{\pi t^2}}\geq {\beta_0}$.

Our main result can be stated as follows:

\begin{theorem}\label{thm1.1}
Assume that $f$ satisfies $(f_1)$-$(f_6)$, then problem $(\ref{problem})$ has at least one ground state solution.
\end{theorem}

\begin{remark}
{\rm A typical example satisfying $(f_1)$-$(f_6)$ is
\begin{align*}
f(t)=|t|^{p-2}te^{\pi t^2},\quad \forall\ p>\max\{3,\kappa+1,\theta\}.
\end{align*}}
\end{remark}

This paper is organized as follows. Section \ref{sec preliminaries} contains some preliminaries. In Section \ref{vf}, we give the variational framework of problem \eqref{problem}.  Section \ref{minimax} is devoted to establish an upper estimation of the ground state energy. The monotonicity of the ground state energy with respect to the mass is studied in Section \ref{mono}. In Section \ref{ps}, we use the minimax principle based on the homotopy stable family to construct a bounded $(PS)$ sequence. Finally, in Section \ref{proof}, we give the proof of Theorem \ref{thm1.1}. Throughout this paper, we will use the notation $\|\cdot\|_q:=\|\cdot\|_{L^q(\mathbb{R}^N)}$, $q\in [1,\infty]$, $B_r(x):=\{y\in \mathbb{R}:|y-x|<r\}$ is the open ball of radius $r$ around $x$, $C,C_i,i\in \mathbb{N}^+$ denote positive constants possibly different from line to line.

\section{{\bfseries Preliminaries}}\label{sec preliminaries}

In this section, we give some preliminaries. For the fractional Laplacian operator, the special case when $s=1/2$ is called the square of the Laplacian. We recall the definition of the fractional Sobolev space
\begin{equation*}
H^{1/2}(\mathbb{R})= \Big\{ u \in L^2(\mathbb{R}): \int_{\mathbb{R}}\int_{\mathbb{R}}\frac{|u(x)-u(y)|^2}{|x-y|^2}\mathrm{d}x\mathrm{d}y<\infty \Big\},
\end{equation*}
endowed with the standard norm
\begin{align*}
\|u\|_{1/2}=\Big(\frac{1}{2\pi}[u]_{1/2}^2+\int_{\mathbb{R}}|u|^2\mathrm{d}x\Big)^{1/2},
\end{align*}
where the term
\begin{equation*}
	[u]_{1/2}=\Big(\int_{\mathbb{R}}\int_{\mathbb{R}}\frac{|u(x)-u(y)|^2}{|x-y|^{2}}\,\mathrm{d}x\mathrm{d}y\Big)^{1/2}
\end{equation*}
denotes the Gagliardo semi-norm of a function $u$. Moreover, by \cite [Proposition 3.6]{di}, we have
\begin{align*}
\|(-\Delta)^{1/4}u\|_2^2=\frac{1}{2\pi}\int_{\mathbb{R}}\int_{\mathbb{R}}\frac{|u(x)-u(y)|^2}{|x-y|^{2}}\,\mathrm{d}x\mathrm{d}y\ \ \mbox{for all}\ \ u\in H^{1/2}(\mathbb{R}).
\end{align*}

Next, we recall the Hardy-Littlewood-Sobolev inequality.
\begin{proposition}\label{hardy}
\cite[Theorem 4.3]{Lieb} Let $1<r,t<\infty$ and $0<\mu<N$ with $\frac{1}{r}+\frac{1}{t}+\frac{\mu}{N}=2$. If $f\in L^r(\mathbb{R}^N)$ and $h\in L^t(\mathbb{R}^N)$, then there exists a sharp constant $C(N,\mu,r,t)>0$ such that
\begin{align}\label{HLS}
\int_{\mathbb{R}^N}\int_{\mathbb{R}^N}\frac{f(x)h(y)}{|x-y|^\mu}\mathrm{d}x\mathrm{d}y\leq C(N,\mu,r,t)\|f\|_r\|h\|_t.
\end{align}
\end{proposition}

\begin{lemma}(Cauchy-Schwarz type inequality) \cite{Matt}
For $g,h\in L_{loc}^1(\mathbb{R}^N)$, there holds
\begin{equation}\label{CS}
  \int_{\mathbb R^{N}}(I_{\mu}\ast |g(x)|)|h(x)|\mathrm{d}x\leq \Big(\int_{\mathbb R^{N}}(I_{\mu}\ast |g(x)|)|g(x)|\mathrm{d}x\Big)^{\frac{1}{2}}\Big(\int_{\mathbb R^{N}}(I_{\mu}\ast |h(x)|)|h(x)|\mathrm{d}x\Big)^{\frac{1}{2}}.
\end{equation}
\end{lemma}

\begin{lemma}\label{GN}
(The fractional Gagliardo-Nirenberg-Sobolev inequality) \cite{frank} Let $u\in H^s(\mathbb{R}^N)$ and $p\in [2,\frac{2N}{N-2s})$, then there exists a sharp constant $C(N,s,p)>0$ such that
\begin{align}\label{gns}
\int_{\mathbb{R}^N}|u|^p\mathrm{d}x\leq C(N,s,p)\Big( \int_{\mathbb{R}^N} |(-\Delta)^{\frac{s}{2}}u|^2\mathrm{d}x
\Big)^{\frac{N(p-2)}{4s}}\Big(
\int_{\mathbb{R}^N} |u|^2\mathrm{d}x
\Big)^{\frac{p}{2}-\frac{N(p-2)}{4s}}.
\end{align}
\end{lemma}

\begin{lemma}\label{tm}
(Full range Adachi-Tanaka-type on $H^{1/2}(\mathbb{R})$) \cite{Taka} It holds that
\begin{align}\label{att}
  \sup_{u\in H^{1/2}(\mathbb{R})\backslash \{0\},\|(-\Delta)^{1/4}u\|_2 \leq1}\frac{1}{\|u\|_2^2}\int_{\mathbb{R}}(e^{\alpha|u|^2}-1)\mathrm{d}x
  \begin{cases}
  <\infty, \quad \alpha<\pi,\\
  =\infty, \quad \alpha\geq\pi.
  \end{cases}
 \end{align}
\end{lemma}

\begin{lemma} \cite[Lemma 4.8]{Kavian}
Let $\Omega \subseteq \mathbb{R}$ be any open set. For $1<s<\infty$, let $\{u_n\}$ be bounded in $L^s(\Omega)$ and $u_n(x)\to u(x)$ a.e. in $\Omega$. Then $u_n(x) \rightharpoonup u(x)$ in $L^s(\Omega)$.
\end{lemma}

  \section{{\bfseries The variational framework}}\label{vf}

Equation $(\ref{problem})$ has a variational structure and its associated energy functional $J: H^{1/2}(\mathbb{R})\to\mathbb{R}$ is defined by
 \begin{align*}
 J(u)=\frac{1}{2}\|(-\Delta)^{1/4}u  \|_2^2-\frac{1}{2}\int_{\mathbb{R}}(I_\mu*F(u))F(u)\mathrm{d}x.
 \end{align*}
  By using assumptions $(f_1)$ and $(f_2)$, it follows that for each $\zeta>0$, $q>1$ and $\alpha>\pi$, there exists $C>0$ such that
\begin{align*}
|f(t)|\leq \zeta |t|^{\kappa}+C|t|^{q-1}(e^{\alpha t^2}-1) \quad \mbox{for all} \ \ t\in\mathbb{R},
\end{align*}
and using $(f_3)$, we have
\begin{align}\label{Ft}
|F(t)|\leq \zeta |t|^{\kappa+1}+C|t|^{q} (e^{\alpha t^2}-1) \quad \mbox{for all} \ \ t\in\mathbb{R}.
\end{align}
By \eqref{HLS} and \eqref{Ft}, we know $J$ is well defined in $H^{1/2}(\mathbb{R})$ and $J\in C^1(H^{1/2}(\mathbb{R}),\mathbb{R})$ with
\begin{align*}
\langle J'(u),v\rangle=&
\frac{1}{2 \pi}\int_{\mathbb{R}}\int_{\mathbb{R}}\frac{[u(x)-u(y)][v(x)-v(y)]}{|x-y|^{2}}\mathrm{d}x\mathrm{d}y-\int_{\mathbb{R}}(I_\mu*F(u))f(u)v \mathrm{d}x,
\end{align*}
for any $u,v\in H^{1/2}(\mathbb{R})$. Hence, a critical point of $J$ on
\begin{equation*}
  S(a)=\Big\{u\in H^{1/2}(\mathbb{R}): \int_{\mathbb{R}}|u|^2 \mathrm{d}x=a^2\Big\}.
\end{equation*}
corresponds to a solution of \eqref{problem}.

To understand the geometry of $J|_{S(a)}$, for any $\beta\in \mathbb{R}$ and $u\in H^{1/2}(\mathbb{R})$, we define
\begin{equation*}
  \mathcal{H}(u,\beta)(x):=e^{\frac{\beta}{2}}u(e^\beta x),\quad \text{for a.e. $x\in \mathbb{R}$}.
\end{equation*}
One can easily check that $\| \mathcal{H}(u,\beta)\|_2=\|u\|_2$ for any $\beta\in \mathbb{R}$ and $ \mathcal{H}(u,\beta_1+\beta_2)=\mathcal{H}(\mathcal{H}(u,\beta_1),\beta_2)=\mathcal{H}(\mathcal{H}(u,\beta_2),\beta_1)$ for any $\beta_1,\beta_2\in \mathbb{R}$.
By Lemma \ref{mountain}, we find that $J$ is unbounded from below on $S(a)$.
 It is well known that all critical points of $J|_{S(a)}$ belong to the Poho\u{z}aev manifold (see \cite{dSS,MS2})
   \begin{align*}
   \mathcal{P}(a)=\big\{u\in S(a): P(u)=0\big\},
   \end{align*}
where
  \begin{align*}
  P(u)=\|(-\Delta)^{1/4}u  \|_2^2+
  (2-\mu)\int_{\mathbb{R}}(I_\mu*F(u))F(u)\mathrm{d}x -\int_{\mathbb{R}}(I_\mu*F(u))f(u)u\mathrm{d}x.
  \end{align*}
This enlightens us to consider the minimization of $J$ on $\mathcal{P}(a)$, i.e.,
\begin{align*}
   m(a)=\inf_{u\in\mathcal{P}(a)}J(u).
   \end{align*}
Our task is to show that $m(a)$ is a critical level of $J|_{S(a)}$. As will be shown in Lemma \ref{pa}, $\mathcal{P}(a)$ is nonempty, thus any critical point $u$ of $J|_{S(a)}$ with $J(u)=m(a)$ is a ground state solution of problem \eqref{problem}.

With a similar argument of \cite[Lemma 3.5]{BS2}, we have the following proposition.
\begin{proposition}\label{m}
Assume that $u_n\to u$ in $H^{1/2}(\mathbb{R})$ and $\beta_n\to \beta$ in $\mathbb{R}$ as $n\to\infty$, then $\mathcal{H}(u_n,\beta_n)\to \mathcal{H}(u,\beta)$ in $H^{1/2}(\mathbb{R})$ as $n\to\infty$.
\end{proposition}

\begin{lemma}   \label{f}
Assume that $(f_1)$-$(f_4)$ hold, let $\{u_n\}\subset S(a)$ be a bounded sequence in
$H^{1/2}(\mathbb{R})$,
  if $u_n\rightharpoonup u$ in $H^{1/2}(\mathbb{R})$ and
  \begin{equation*}
  \int_{\mathbb{R}}(I_\mu*F(u_n))f(u_n)u_n\mathrm{d}x\leq K_0
 \end{equation*}
 for some $K_0>0$, then for any $\phi\in C_0^\infty(\mathbb{R})$, we have
 \begin{align*}
 \int_{\mathbb{R}}(I_\mu*F(u_n))f(u_n)\phi\mathrm{d}x\to \int_{\mathbb{R}}(I_\mu*F(u))f(u)\phi\mathrm{d}x,\quad \mbox{as $n\rightarrow\infty$}.
 \end{align*}
\end{lemma}

 \begin{proof}
 By Fatou Lemma, we have
 \begin{align*}
 \int_{\mathbb{R}}(I_\mu*F(u))f(u)u\mathrm{d}x\leq K_0.
  \end{align*}
 Denote $\Omega:=supp\phi$. For any given $\varepsilon>0$, let $M_\varepsilon:=\frac{K_0\|\phi\|_{\infty}}{\varepsilon}$, then by $(f_3)$, we obtain
 \begin{align}\label{27}
 \int_{\{|u_n|\geq M_\varepsilon\}\cup\{|u|=M_\varepsilon\}}(I_\mu*F(u_n))\big|f(u_n)\phi\big| \mathrm{d}x\leq\frac{2\varepsilon}{K_0}
 \int_{|u_n|\geq \frac{M_\varepsilon}{2}}(I_\mu*F(u_n))f(u_n)u_n\mathrm{d}x\leq 2\varepsilon
   \end{align}
 and
  \begin{align}\label{28}
 \int_{|u|\geq M_\varepsilon}(I_\mu*F(u)) \big|f(u)\phi\big| \mathrm{d}x\leq \frac{\varepsilon}{K_0}
 \int_{|u|\geq M_\varepsilon}(I_\mu*F(u))f(u)u\mathrm{d}x\leq \varepsilon.
  \end{align}
Let $G_\varepsilon:=\{x\in\Omega:|u(x)|=M_\varepsilon\}$, since $|f(u_n)|{\chi_{|u_n|\leq M_\varepsilon}}\to |f(u)|{\chi_{|u|\leq M_\varepsilon}}$ a.e. in $\Omega\backslash G_\varepsilon$, and for any $x\in\Omega$, we have
\begin{equation*}
|f(u_n)|{\chi_{|u_n|\leq M_\varepsilon}}\leq \max\limits_{|t|\leq M_\varepsilon}|f(t)|< \infty,
\end{equation*}
 using the Lebesgue dominated convergence theorem, we obtain
\begin{align}\label{22}
\lim_{n\to\infty}\int_{\{\Omega\backslash G_\varepsilon\}\cap\{|u_n|\leq M_\varepsilon\}}|f(u_n)|^{\frac{2}{2-\mu}}\mathrm{d}x
=\int_{\{\Omega\backslash G_\varepsilon\}\cap\{|u|\leq M_\varepsilon\}}|f(u)|^{\frac{2}{2-\mu}}\mathrm{d}x.
\end{align}
Choosing $L_\varepsilon>t_0$ such that
\begin{align}\label{23}
\|\phi\|_{\infty}(\frac{M_0K_0}{L_\varepsilon})^{\frac{1}{2}}\Big(\int_{\Omega}|f(u)|^{\frac{2}{2-\mu}}\mathrm{d}x\Big)^{\frac{2-\mu}{2}}<\varepsilon
\end{align}
 and
 \begin{align}\label{26}
 \int_{|u|\leq M_\varepsilon}(I_\mu*F(u){\chi_{|u|\geq L_\varepsilon}})|f(u)\phi|\mathrm{d}x<\varepsilon.
 \end{align}
Then from $(f_4)$, \eqref{HLS}, \eqref{CS}, $(\ref{22})$ and $(\ref{23})$, one has
\begin{align}\label{25}
& \int_{\{|u_n|\leq M_\varepsilon\}\cap\{|u|\neq M_\varepsilon\}}(I_\mu*F(u_n){\chi_{|u_n|\geq L_\varepsilon}})|f(u_n)\phi|\mathrm{d}x\nonumber\\
 \leq& \|\phi\|_{\infty} \int_{\Omega\backslash G_\varepsilon}(I_\mu*F(u_n){\chi_{|u_n|\geq L_\varepsilon}})|f(u_n)|{\chi_{|u_n|\leq M_\varepsilon}}\mathrm{d}x\nonumber\\
  \leq&|\phi\|_{\infty} \Big(\int_{\mathbb{R}}(I_\mu*F(u_n){\chi_{|u_n|\geq L_\varepsilon}})|F(u_n)|{\chi_{|u_n|\geq L_\varepsilon}}\mathrm{d}x\Big)^{\frac{1}{2}}\nonumber\\
&\times \Big(\int_{\mathbb{R}}(I_\mu*|f(u_n)|\chi_{\{\Omega\backslash G_\varepsilon\}\cap\{|u_n|\leq M_\varepsilon\}})|f(u_n)|{\chi_{\{\Omega\backslash G_\varepsilon\}  \cap\{|u_n|\leq M_\varepsilon\}}}  \mathrm{d}x\Big)^{\frac{1}{2}}\nonumber\\
\leq &C\|\phi\|_{\infty} \Big(\int_{|u_n|\geq L_\varepsilon}(I_\mu*F(u_n))F(u_n)\mathrm{d}x\Big)^{\frac{1}{2}}
\Big( \int_{\{\Omega\backslash G_\varepsilon\}\cap\{|u_n|\leq M_\varepsilon\}}|f(u_n)|^{\frac{2}{2-\mu}}\mathrm{d}x\Big)^{\frac{2-\mu}{2}}\nonumber\\
\leq &C\|\phi\|_{\infty} \Big(\frac{M_0}{L_\varepsilon}\int_{|u_n|\geq L_\varepsilon}(I_\mu*F(u_n))f(u_n)u_n\mathrm{d}x\Big)^{\frac{1}{2}}
\Big( \int_{\Omega}|f(u)|^{\frac{2}{2-\mu}}\mathrm{d}x\Big)^{\frac{2-\mu}{2}}\nonumber\\
\leq &C\|\phi\|_{\infty} \Big(\frac{K_0M_0}{L_\varepsilon}\Big)^{\frac{1}{2}}
\Big( \int_{\Omega}|f(u)|^{\frac{2}{2-\mu}}\mathrm{d}x\Big)^{\frac{2-\mu}{2}}< \varepsilon.
      \end{align}
For any $x\in\mathbb{R}$, we set
\begin{align*}
d_n(x):=\int_{\mathbb{R}}\frac{|F(u_n)|{\chi_{|u_n|\leq L_\varepsilon}}}{|x-y|^\mu}\mathrm{d}y\quad \mbox{and}\quad
d(x):=\int_{\mathbb{R}}\frac{|F(u)|{\chi_{|u|\leq L_\varepsilon}}}{|x-y|^\mu}\mathrm{d}y.
\end{align*}
According to $(\ref{Ft})$, for any $x\in \mathbb{R}$ and $R>0$, one has
\begin{align*}
&|d_n(x)-d(x)|\\
\leq& \int_{\mathbb{R}}\frac{\Big||F(u_n)|{\chi_{|u_n|\leq L_\varepsilon}}-|F(u)|{\chi_{|u|\leq L_\varepsilon}}\Big|}{|x-y|^\mu}\mathrm{d}y\\
\leq&\Big(\int_{|x-y|\leq R}
\Big||F(u_n)|{\chi_{|u_n|\leq L_\varepsilon}}-|F(u)|{\chi_{|u|\leq L_\varepsilon}}\Big|^{\frac{1+\mu}{1-\mu}}\mathrm{d}y\Big)^{\frac{1-\mu}{1+\mu}}
\Big(\int_{|x-y|\leq R}|x-y|^{-\frac{1+\mu}{2}}\mathrm{d}y\Big)^{\frac{2\mu}{1+\mu}}\\
&+\Big(\int_{|x-y|\geq R}
\Big||F(u_n)|{\chi_{|u_n|\leq L_\varepsilon}}-|F(u)|{\chi_{|u|\leq L_\varepsilon}}\Big|^{\frac{2-\mu}{2-2\mu}}\mathrm{d}y\Big)^{\frac{2-2\mu}{2-\mu}}
\Big(\int_{|x-y|\geq R}|x-y|^{\mu-2}\mathrm{d}y\Big)^{\frac{\mu}{2-\mu}}\\
\leq&\Big(\frac{4}{1-\mu}R^{\frac{1-\mu}{2}}\Big)^{\frac{2\mu}{1+\mu}}\Big(\int_{|x-y|\leq R}
\Big||F(u_n)|{\chi_{|u_n|\leq L_\varepsilon}}-|F(u)|{\chi_{|u|\leq L_\varepsilon}}\Big|^{\frac{1+\mu}{1-\mu}}\mathrm{d}y\Big)^{\frac{1-\mu}{1+\mu}}\\
&+\Big(\frac{2}{(1-\mu)R^{1-\mu}}\Big)^{\frac{\mu}{2-\mu}}
\Big(\int_{|x-y|\geq R}
\Big||F(u_n)|{\chi_{|u_n|\leq L_\varepsilon}}-|F(u)|{\chi_{|u|\leq L_\varepsilon}}\Big|^{\frac{2-\mu}{2-2\mu}}\mathrm{d}y\Big)^{\frac{2-2\mu}{2-\mu}}.
\end{align*}
 Similar to $\eqref{22}$, we can get
 \begin{align*}
 \int_{|x-y|\leq R}
\Big||F(u_n)|{\chi_{|u_n|\leq L_\varepsilon}}-|F(u)|{\chi_{|u|\leq L_\varepsilon}}\Big|^{\frac{1+\mu}{1-\mu}}\mathrm{d}y\to0,
\quad\mbox{as} \ n\to\infty.
\end{align*}
Moreover, it follows from $(\ref{Ft})$ that
 \begin{align*}
 &\Big(\int_{|x-y|\geq R}
\Big||F(u_n)|{\chi_{|u_n|\leq L_\varepsilon}}-|F(u)|{\chi_{|u|\leq L_\varepsilon}}\Big|^{\frac{2-\mu}{2-2\mu}}\mathrm{d}y\Big)^{\frac{2-2\mu}{2-\mu}}
\\&\leq (\|u_n\|_{\frac{(\kappa+1)(2-\mu)}{2-2\mu}}^{\kappa+1}+\|u_n\|_{\frac{q(2-\mu)}{2-2\mu}}^q+\|u\|_{\frac{(\kappa+1)(2-\mu)}{2-2\mu}}^{\kappa+1}+\|u\|_{\frac{q(2-\mu)}{2-2\mu}}^q\Big) \leq C.
 \end{align*}
Choosing $R>0$ large enough, then for any $x\in\mathbb{R}$, we obtain $d_n(x)\to d(x)$ as $n\rightarrow\infty$.
Using $(\ref{Ft})$ again, for any $x\in \mathbb{R}$, we also have
\begin{align*}
\begin{split}
d_n(x)\leq&\Big(\int_{|x-y|\leq R}
\Big||F(u_n)|{\chi_{|u_n|\leq L_\varepsilon}}\Big|^{\frac{1+\mu}{1-\mu}}\mathrm{d}y\Big)^{\frac{1-\mu}{1+\mu}}
\Big(\int_{|x-y|\leq R}|x-y|^{-\frac{1+\mu}{2}}\mathrm{d}y\Big)^{\frac{2\mu}{1+\mu}}\\
&+\Big(\int_{|x-y|\geq R}
\Big||F(u_n)|{\chi_{|u_n|\leq L_\varepsilon}}\Big|^{\frac{2-\mu}{2-2\mu}}\mathrm{d}y\Big)^{\frac{2-2\mu}{2-\mu}}
\Big(\int_{|x-y|\geq R}|x-y|^{\mu-2}\mathrm{d}y\Big)^{\frac{\mu}{2-\mu}}\\
\leq&\Big(\frac{4}{1-\mu}R^{\frac{1-\mu}{2}}\Big)^{\frac{2\mu}{1+\mu}}(2R)^{\frac{1-\mu}{1+\mu}}\max_{|t|\leq L_\varepsilon}|F(t)|\\
&+C\Big(\frac{2}{(1-\mu)R^{1-\mu}}\Big)^{\frac{\mu}{2-\mu}}
\Big(\|u_n\|_{\frac{(\kappa+1)(2-\mu)}{2-2\mu}}^{\kappa+1}+\|u_n\|_{\frac{q(2-\mu)}{2-2\mu}}^q\Big)
\leq C.
\end{split}
\end{align*}
Thus, we have
\begin{align*}
\big|d_n(x)f(u_n(x)){\chi_{|u_n|\leq M_\varepsilon}}\phi(x)\big|\leq C\|\phi\|_{\infty}\max_{|t|\leq M_\varepsilon}|f(t)|,
\end{align*}
 for any $x\in\Omega$. This inequality together with $d_n(x)\to d(x)$, and the Lebesgue dominated convergence theorem, yields that
 \begin{align}\label{24}
\lim_{n\to\infty}\int_{\{|u_n|\leq M_\varepsilon\}\cap\{|u|\neq M_\varepsilon\}}
(I_\mu*F(u_n){\chi_{|u_n|\leq L_\varepsilon}})|f(u_n)\phi|\mathrm{d}x
=\int_{|u|\leq M_\varepsilon}(I_\mu*F(u){\chi_{|u|\leq L_\varepsilon}})|f(u)\phi|\mathrm{d}x.
 \end{align}
 Combining $(\ref{27})$, $(\ref{28})$ and $(\ref{26})$-$(\ref{24})$, we complete the proof.
 \end{proof}

\section{{\bfseries The estimation for the upper bound of $m(a)$}}\label{minimax}

In this section, by using the condition $(f_6)$, we obtain an upper bound of $m(a)$.

\begin{lemma}\label{mountain}
Assume that $(f_1)$-$(f_3)$ hold. Let $u\in S(a)$ be arbitrary but fixed, then we have

(i)  $J(\mathcal{H}(u,\beta))\to0^+$ as $\beta\to -\infty$; 

(ii)  $J(\mathcal{H}(u,\beta))\to -\infty$ as $\beta\to +\infty$.
\end{lemma}

\begin{proof}
$(i)$ By a straightforward calculation, we have
\begin{align*}
\int_{\mathbb{R}}|\mathcal{H}(u,\beta) |^2 \mathrm{d}x=a^2,\ \ \int_{\mathbb{R}}|\mathcal{H}(u,\beta) |^\xi \mathrm{d}x=e^{\frac{(\xi-2)\beta}{2}}\int_{\mathbb{R}}|u|^\xi\mathrm{d}x, \ \ \forall \ \xi>2,
\end{align*}
and
\begin{align*}
\int_{\mathbb{R}}\int_{\mathbb{R}}\frac{|\mathcal{H}(u,\beta)(x)-\mathcal{H}(u,\beta)(y)|^2}{|x-y|^2}
\mathrm{d}x\mathrm{d}y
=e^{\beta}\int_{\mathbb{R}}\int_{\mathbb{R}}\frac{| u(x)-u(y)|^2}{|x-y|^2}
\mathrm{d}x\mathrm{d}y.
\end{align*}
Thus there exist $\beta_1<<0$  such that $\|(-\Delta)^{1/4}\mathcal{H}(u,\beta)  \|_{2}^2<  \frac{2-\mu}{2}
$ for any $\beta<\beta_1$. Fix $\alpha>\pi$ close to $\pi$ and $\nu>1$ close to $1$ such that
\begin{equation*}
  \frac{2\alpha \nu}{2-\mu}\|(-\Delta)^{1/4}\mathcal{H}(u,\beta)  \|_{2}^2<\pi,\quad \text{for any $\beta<\beta_1$}.
\end{equation*}
Then, for $\frac{1}{\nu}+\frac{1}{\nu'}=1$, using  \eqref{att}, $(\ref{Ft})$, the H\"{o}lder and Sobolev inequality, we have
\begin{align}\label{tain1}
&\|F(\mathcal{H}(u,\beta))\|_{\frac{2}{2-\mu}} \leq \|f(\mathcal{H}(u,\beta))\mathcal{H}(u,\beta)\|_{\frac{2}{2-\mu}}\nonumber\\
\leq&
\zeta \|\mathcal{H}(u,\beta) \|_{\frac{2(\kappa+1)}{2-\mu}}^{\kappa+1}+C \Big[\int_{\mathbb{R}} \big[(e^{\alpha |\mathcal{H}(u,\beta)|^2}-1) |\mathcal{H}(u,\beta)|^{q}\big]^{\frac{2}{2-\mu}} \mathrm{d}x\Big]^{\frac{2-\mu}{2}}\nonumber\\
\leq&\zeta \|\mathcal{H}(u,\beta) \|_{\frac{2(\kappa+1)}{2-\mu}}^{\kappa+1}+C
\Big[\int_{\mathbb{R}}(e^{\frac{2 \alpha \nu}{2-\mu}|\mathcal{H}(u,\beta)|^2}-1)\mathrm{d}x\Big]^{\frac{2-\mu}{2\nu}}\| \mathcal{H}(u,\beta) \|_{\frac{2q\nu'}{2-\mu}}^q\nonumber\\
=&\zeta \|\mathcal{H}(u,\beta) \|_{\frac{2(\kappa+1)}{2-\mu}}^{\kappa+1}+C
\Big[\int_{\mathbb{R}}(e^{\frac{2 \alpha \nu}{2-\mu}\|(-\Delta)^{1/4}\mathcal{H}(u,\beta)\|_{2}^2 \big(\frac{|\mathcal{H}(u,\beta)|}{\|(-\Delta)^{1/4}\mathcal{H}(u,\beta)\|_{2}}\big)^2}-1)\mathrm{d}x\Big]^{\frac{2-\mu}{2\nu}}\| \mathcal{H}(u,\beta) \|_{\frac{2q\nu'}{2-\mu}}^q\nonumber\\
\leq& \zeta \|\mathcal{H}(u,\beta) \|_{\frac{2(\kappa+1)}{2-\mu}}^{\kappa+1}+C\| \mathcal{H}(u,\beta) \|_{{\frac{2q\nu'}{2-\mu}}}^q\nonumber\\
=&\zeta e^{\frac{(\kappa+\mu-1)\beta}{2}}
\| u \|_{\frac{2(\kappa+1)}{2-\mu}}^{\kappa+1}+Ce^{\frac{(q\nu'+\mu-2)\beta}{2\nu'}}
\| u\|_{{\frac{2q\nu'}{2-\mu}}}^q.
\end{align}
Hence, it follows from \eqref{HLS} that
\begin{align*}
\int_{\mathbb{R}}\big(I_\mu*F(\mathcal{H}(u,\beta) )\big)F(\mathcal{H}(u,\beta) )\mathrm{d}x& \leq C \|\mathcal{H}(u,\beta) \|_{\frac{2(\kappa+1)}{2-\mu}}^{2(\kappa+1)}+C\| \mathcal{H}(u,\beta) \|_{{\frac{2q\nu'}{2-\mu}}}^{2q}\\
&\leq C e^{(\kappa+\mu-1)\beta}
\|u \|_{{\frac{2(\kappa+1)}{2-\mu}}}^{2(\kappa+1)}+Ce^{\frac{(q\nu'+\mu-2)\beta}{\nu'}}
\| u \|_{{\frac{2q\nu'}{2-\mu}}}^{2q}.
\end{align*}
Since $\kappa>2-\mu$, $q>1$ and $\nu'$ large enough, it follows that
\begin{align*}
 J(\mathcal{H}(u,\beta))\geq \frac{1}{2}e^\beta \|(-\Delta)^{1/4}u\|_2^2- C e^{(\kappa+\mu-1)\beta}
\|u \|_{{\frac{2(\kappa+1)}{2-\mu}}}^{2(\kappa+1)}-Ce^{\frac{(q\nu'+\mu-2)\beta}{\nu'}}
\| u \|_{{{\frac{2q\nu'}{2-\mu}}}}^{2q}  \to0^+,\ \ \mbox{as} \ \beta\to-\infty.
\end{align*}

$(ii)$
For any fixed $\beta>>0$,  set
\begin{align*}
\mathcal{W}(t):=\frac{1}{2}\int_{\mathbb{R}}(I_\mu*F(tu))F(tu)\mathrm{d}x\quad \text{for $t>0$}.
\end{align*}
Using $(f_3)$, one has
\begin{align*}
\frac{\frac{\mathrm{d}\mathcal{W}(t)}{\mathrm{d}t}}{\mathcal{W}(t)}>\frac{2\theta}{t}\quad \text{for $t>0$}.
\end{align*}
Thus, integrating this over $[1,e^{\frac{\beta}{2}}]$, we get
\begin{align}\label{fff}
\int_{\mathbb{R}}(I_\mu*F(e^{\frac{\beta}{2}}u))F(e^{\frac{\beta}{2}}u)\mathrm{d}x\geq e^{\theta\beta}\int_{\mathbb{R}}(I_\mu*F(u))F(u)\mathrm{d}x.
\end{align}
Hence,
\begin{align*}
J(\mathcal{H}(u,\beta))\leq \frac{1}{2}e^\beta \|(-\Delta)^{1/4}u\|_2^2- \frac{1}{2} e^{(\theta+\mu-2)\beta}\int_{\mathbb{R}}(I_\mu*F(u))F(u)\mathrm{d}x.
\end{align*}
Since $\theta>3-\mu$, the above inequality yields that $J(\mathcal{H}(u,\beta))\to -\infty$ as $\beta\to+\infty$.
\end{proof}

\begin{lemma}\label{pa}
Assume that $(f_1)$-$(f_3)$ and $(f_5)$ hold. Then for any fixed $u\in S(a)$, the function $I_u(\beta):=J(\mathcal{H}(u,\beta))$ reaches its unique maximum with positive level at a unique point $\beta_u\in\mathbb{R}$ such that $\mathcal{H}(u,\beta_u)\in\mathcal{P}(a)$. Moreover, the mapping $u\rightarrow \beta_u$ is continuous in $u\in S(a)$.
\end{lemma}

\begin{proof}
From Lemma $\ref{mountain}$, there exists $\beta_u\in\mathbb{R}$ such that $P(\mathcal{H}(u,\beta_u))=\frac{\mathrm{d}}{\mathrm{d}\beta}I_u(\beta)\Big|_{\beta=\beta_u}=0$ and $J(\mathcal{H}(u,\beta_u))>0$. Next, we prove the uniqueness of $\beta_u$.
For $u\in S(a)$ and $\beta\in\mathbb{R}$, we know
\begin{align*}
I_u(\beta)=\frac{1}{2}e^\beta \|(-\Delta)^{1/4}u\|_{2}^2-\frac{1}{2} e^{(\mu-2)\beta} \int_{\mathbb{R}}  (I_\mu*F(e^{\frac{\beta}{2}}u))F(e^{\frac{\beta}{2}} u)\mathrm{d}x,
\end{align*}
and
\begin{align*}
P(\mathcal{H}(u,\beta))=\frac{\mathrm{d}}{\mathrm{d}\beta}I_u(\beta)
=&e^\beta\Big(
\frac{1}{2} \|(-\Delta)^{1/4}u\|_2^2+\frac{2-\mu}{2} e^{(\mu-3)\beta} \int_{\mathbb{R}}  (I_\mu*F(e^{\frac{\beta}{2}}u))F(e^{\frac{\beta}{2}} u)\mathrm{d}x\\
&\ \ \ - \frac{1}{2}e^{(\mu-3)\beta} \int_{\mathbb{R}}  (I_\mu*F(e^{\frac{\beta}{2}}u))f(e^{\frac{\beta}{2}} u)e^{\frac{\beta}{2}} u\mathrm{d}x\Big)\\
=&\frac{1}{2}e^\beta\Big( \|(-\Delta)^{1/4}u\|_{2}^2-\Phi(\beta)
\Big),
\end{align*}
where
\begin{align*}
\Phi(\beta)=\int_{\mathbb{R}} \Big(I_\mu*\frac{F(e^{\frac{\beta}{2}}u)}{(e^{\frac{\beta}{2}})^{3-\mu}} \Big)
\frac{\widetilde F(e^{\frac{\beta}{2}}u)}{(e^{\frac{\beta}{2}})^{3-\mu}}\mathrm{d}x.
\end{align*}
For any $t\in\mathbb{R} \backslash \{0\}$, using $(f_3)$ and $(f_5)$, we see that $\frac{F(\beta t)}{\beta^{3-\mu}}$ is strictly increasing in $\beta \in (0,+\infty)$ and $\frac{\widetilde F(\beta t)}{\beta^{3-\mu}}$ is non-decreasing in $\beta\in(0,+\infty)$. This implies that $\Phi(\beta)$ is strictly increasing in $\beta\in(0,+\infty)$ and there is at most one $\beta_u\in\mathbb{R}$ such that $\mathcal{H}(u,\beta_u)\in\mathcal{P}(a)$.

 From the above arguments, the mapping $u\rightarrow \beta_u$ is well defined. Let $\{u_n\}\subset S(a)$ be a sequence such that $u_n\to u$ in $H^{1/2}(\mathbb{R})$ as $n\to \infty$. We only need to prove that, up to a subsequence, $\beta_{u_n}\to \beta_u$ in $\mathbb{R}$ as $n\to \infty$.

On the one hand, if up to a subsequence, $\beta_{u_n}\to+\infty$ as $n\to\infty$, then by \eqref{fff} and $u_n\to u\neq 0$ in $H^{1/2}(\mathbb{R})$ as $n\to \infty$, we have
\begin{align*}
0\leq\lim_{n\to\infty}e^{-\beta_{u_n}}J(\mathcal{H}(u_n,\beta_{u_n}))
&\leq\lim_{n\to\infty} \frac{1}{2}\Big[ \|(-\Delta)^{1/4}u_n\|_{2}^2
-e^{(\theta+\mu-3)\beta_{u_n}} \int_{\mathbb{R}}  (I_\mu*F(u_n))F(u_n)\mathrm{d}x\Big]\\
&=-\infty,
\end{align*}
which is a contradiction. Hence, $\{\beta_{u_n}\}$ is bounded from above.

On the other hand, by Proposition $\ref{m}$,
we know $\mathcal{H}(u_n,\beta_u)\to \mathcal{H}(u,\beta_u)$ in $H^{1/2}(\mathbb{R})$ as $n\to\infty$. Then
\begin{align*}
J(\mathcal{H}(u_n,\beta_{u_n}))\geq J(\mathcal{H}(u_n,\beta_u))=J(\mathcal{H}(u,\beta_u))+o_n(1),
\end{align*}
and thus
\begin{align*}
\liminf_{n\to\infty}J(\mathcal{H}(u_n,\beta_{u_n}))\geq J(\mathcal{H}(u,\beta_u))>0.
\end{align*}
If up to a subsequence, $\beta_{u_n}\to -\infty$ as $n\to\infty$, using $(f_3)$, we get
\begin{align*}
J(\mathcal{H}(u_n,\beta_{u_n}))\leq \frac{e^{\beta_{u_n}}}{2}\|(-\Delta)^{1/4}u_n\|_2^2\to0, \quad \text{as $n\rightarrow\infty$},
\end{align*}
which is impossible. So we get $\{\beta_{u_n}\}$ is bounded from below. Up to a subsequence, we assume that $\beta_{u_n}\to \beta_0$ as $n\to\infty$. Since $u_n\to u$ in $H^{1/2}(\mathbb{R})$, then $\mathcal{H}(u_n,\beta_{u_n})\to\mathcal{H}(u,\beta_0)$ in $H^{1/2}(\mathbb{R})$ as $n\to\infty$. Moreover, by $P(\mathcal{H}(u_n,\beta_{u_n}))=0$, it follows that $P(\mathcal{H}(u,\beta_0))=0$. By the uniqueness of $\beta_u$, we get $\beta_u=\beta_0$ and the conclusion follows.
\end{proof}

\begin{lemma}\label{ous11}
Assume that $(f_1)$-$(f_3)$ hold, then there exists $\gamma>0$ small enough such that
\begin{align*}
J(u)\geq \frac{1}{4}\|(-\Delta)^{1/4}u\|_2^2
\quad \mbox{and}\quad
P(u)\geq\frac{1}{2}\|(-\Delta)^{1/4}u\|_2^2
\end{align*}
for $u\in S(a)$ satisfying $\|(-\Delta)^{1/4}u\|_2\leq \gamma$.
\end{lemma}

\begin{proof}
If $\gamma<\sqrt{\frac{2-\mu}{2}}$, then $\|(-\Delta)^{1/4}u\|_2^2\leq \frac{2-\mu}{2}$. Fix $\alpha>\pi$ close to $\pi$ and $\nu>1$ close to $1$ such that
\begin{align*}
\frac{2\alpha\nu}{2-\mu}\|(-\Delta)^{1/4}u\|_2^2< \pi.
\end{align*}
From \eqref{gns} and $(\ref{tain1})$, we obtain
\begin{align*}
\int_{\mathbb{R}}\big(I_\mu*F(u)\big)F(u)\mathrm{d}x \leq& C \|u \|_{\frac{2(\kappa+1)}{2-\mu}}^{2(\kappa+1)}+C\| u \|_{{\frac{2q\nu'}{2-\mu}}}^{2q}\\
\leq& C a^{2(2-\mu)} \|(-\Delta)^{1/4}u\|_2^{2(\kappa+\mu-1)} +C a^{\frac{2(2-\mu)}{\nu'}} \|(-\Delta)^{1/4}u\|_2^{\frac{2(q\nu'+\mu-2)}{\nu'}}\\
\leq& C (a^{2(2-\mu)}\gamma^{2(\kappa+\mu-2)}+a^{\frac{2(2-\mu)}{\nu'}}\gamma^{2q-2+\frac{2(\mu-2)}{\nu'}}
)\|(-\Delta)^{1/4}u\|_2^2.
\end{align*}
Similarly, we can get
\begin{align*}
\int_{\mathbb{R}}(I_\mu*F(u))f(u)u\mathrm{d}x \leq C (a^{2(2-\mu)}\gamma^{2(\kappa+\mu-2)}+a^{\frac{2(2-\mu)}{\nu'}}\gamma^{2q-2+\frac{2(\mu-2)}{\nu'}}
)\|(-\Delta)^{1/4}u\|_2^2.
\end{align*}
Since $\kappa>2-\mu$, $q>1$ and $\nu'=\frac{\nu}{\nu-1}$ large enough, choosing $0<\gamma<\sqrt{\frac{2-\mu}{2}}$ small enough, we complete the proof.
\end{proof}

\begin{lemma}\label{ous1}
Assume that $(f_1)$-$(f_3)$ and $(f_5)$ hold, then we have $\inf\limits_{u\in\mathcal{P}(a)}\|(-\Delta)^{1/4}u\|_2>0$ and $m(a)>0$.
\end{lemma}

\begin{proof}
By Lemma $\ref{pa}$, we know $\mathcal{P}(a)$ is nonempty. Supposed that there exists a sequence $\{u_n\}\subset \mathcal{P}(a)$ such that $\|(-\Delta)^{1/4}u_n\|_2\to0$ as $n\to\infty$,
then by Lemma $\ref{ous11}$, up to subsequence,
\begin{align*}
0=P(u_n)\geq\frac{1}{2} \|(-\Delta)^{1/4}u_n\|_2^2\geq0,
\end{align*}
which implies that $\|(-\Delta)^{1/4}u_n\|_2^2=0$ for any $n\in\mathbb{N}^+$.
By $(f_3)$ and $P(u_n)=0$, we have
\begin{align*}
0=&(2-\mu)\int_{\mathbb{R}}(I_\mu*F(u_n))F(u_n)\mathrm{d}x-\int_{\mathbb{R}}(I_\mu*F(u_n))f(u_n)u_n\mathrm{d}x\\
\leq& (\frac{2-\mu}{\theta}-1) \int_{\mathbb{R}}(I_\mu*F(u_n))f(u_n)u_n\mathrm{d}x\leq0.
\end{align*}
So $u_n\to0$ a.e. in $\mathbb{R}$, which contradicts  $a>0$.

From Lemma $\ref{pa}$, we know that for any $u\in\mathcal{P}(a)$,
\begin{align*}
J(u)=J(\mathcal{H}(u,0))\geq J(\mathcal{H}(u,\beta)),\quad \forall \ \beta\in\mathbb{R}.
\end{align*}
 Let $\gamma>0$ be the number given by Lemma $\ref{ous11}$ and $e^{\beta}=\frac{\gamma^2}{ \|(-\Delta)^{1/4}u\|_2^2}$, then $ \|(-\Delta)^{1/4}\mathcal{H}(u,\beta)\|_2^2=\gamma^2$. Applying Lemma $\ref{ous11}$ again, we deduce that
\begin{align*}
J(u)\geq J(\mathcal{H}(u,\beta))\geq \frac{1}{4}\|(-\Delta)^{1/4}u\|_2^2
\geq\frac{\gamma^2}{4}>0.
\end{align*}
This completes the proof.
\end{proof}

In order to estimate the upper bound of $m(a)$, let us consider the following sequence of nonnegative functions (see \cite{Taka}) supported in $B_1(0)$ given by
  \begin{align*}
\varpi_n(x)=\frac{1}{\sqrt{\pi}}
	 \begin{cases}
  \sqrt{\log n},\ \  &\mbox{for}\ \ |x|<\frac{1}{n},\\
 \frac{\log{\frac{1}{|x|}}}{\sqrt{\log n}},\ \  &\mbox{for}\ \ \frac{1}{n}\leq |x|\leq 1,\\
  0,\ \  &\mbox{for}\ \ |x|> 1.
      \end{cases}
 \end{align*}
One can check that $\varpi_n\in H^{1/2}(\mathbb{R})$. A direct calculation shows that
\begin{align*}
\|(-\Delta)^{1/4}\varpi_n\|_2^2=1+o(1),
\end{align*}
\begin{align*}
\delta_n:=\|\varpi_n\|_2^2=&\int_{-\frac{1}{n}}^{\frac{1}{n}}\frac{\log n}{\pi}\mathrm{d}x+
\int_{-1}^{-\frac{1}{n}}\frac{(\log|x|)^2}{\pi\log n}\mathrm{d}x+
\int_{\frac{1}{n}}^1\frac{(\log|x|)^2}{\pi\log n}\mathrm{d}x\\
=&\frac{4}{\pi}(\frac{1}{\log n}-\frac{1}{n\log n }-\frac{1}{n})=\frac{4}{\pi\log n}+o(\frac{1}{\log n}).
\end{align*}
Let $\omega_n:=\frac{a\varpi_n}{\|\varpi_n\|_2}$. Then $\omega_n \in S(a)$ and
\begin{align}\label{wnx}
\omega_n(x)=\frac{a}{2}
	 \begin{cases}
  \log n (1+o(1)),\ \  &\mbox{for}\ \ |x|<\frac{1}{n},\\
 \log{\frac{1}{|x|}}(1+o(1)),\ \  &\mbox{for}\ \ \frac{1}{n}\leq |x|\leq 1,\\
  0,\ \  &\mbox{for}\ \ |x|\geq 1.
      \end{cases}
\end{align}
Furthermore, we have
\begin{align}\label{wn}
\|(-\Delta)^{1/4}\omega_n\|_2^2&=\int_{\mathbb{R}}\int_{\mathbb{R}}\frac{\big|\frac{a}{\sqrt{\delta_n}}\varpi_n(x)-\frac{a}{\sqrt{\delta_n}}\varpi_n(y)\big|^2}{|x-y|^2}\mathrm{d}x\mathrm{d}y\nonumber\\
&=\frac{a^2}{\delta_n}\|(-\Delta)^{1/4}\varpi_n\|_2^2=\frac{\pi a^2\log n}{4}(1+o(1)).
\end{align}
For any $t>0$, let
\begin{align}\label{phi}
\Phi_n(t):=J(t\omega_n(t^{2}x))=\frac{t^2}{2}\|(-\Delta)^{1/4}\omega_n\|_2^2-\frac{1}{2}t^{2(\mu-2)}\int_{\mathbb{R}}(I_\mu*F(t\omega_n))F(t\omega_n)\mathrm{d}x.
\end{align}
From Lemmas {\ref{pa}} and $\ref{ous1}$, we infer that $m(a)=\inf\limits_{u\in S(a)}\max\limits_{\beta\in\mathbb{R}}J(\mathcal{H}(u,\beta))>0$, this together with $\omega_n\in S(a)$ yields that
\begin{align*}
m(a)\leq \max_{\beta\in\mathbb{R}}J(\mathcal{H}(\omega_n,\beta))=\max_{t>0}\Phi_n(t).
\end{align*}

\begin{lemma}\label{attain}
Assume that $(f_1)$-$(f_3)$ hold, then for any fixed $n\in \mathbb{N}^+$, $\max\limits_{t\geq0}\Phi_n(t)>0$ is attained at some $t_n>0$.
\end{lemma}
\begin{proof}
For any fixed $n\in \mathbb{N}^+$, as $t>0$ small enough, fix $\alpha>\pi$ close to $\pi$ and $\nu>1$ close to $1$ such that
\begin{align*}
\frac{2\alpha\nu}{2-\mu}\|(-\Delta)^{1/4}(t\omega_n)\|_2^2< \pi.
\end{align*}
Arguing as $(\ref{tain1})$, by \eqref{att}, for $\nu'=\frac{\nu}{\nu-1}$, we have
\begin{align*}
t^{2(\mu-2)}\int_{\mathbb{R}}(I_\mu*F(t\omega_n))F(t\omega_n)\mathrm{d}x \leq&
Ct^{2(\mu-2)}\Big(\|t\omega_n\|_{\frac{2(\kappa+1)}{2-\mu}}^{2(\kappa+1)}+t^{\frac{2(2-\mu)}{\nu}}\|t\omega_n\|_{\frac{2q\nu'}{2-\mu}}^{2q}\Big)\\
=&Ct^{2(\kappa+\mu-1)}\|\omega_n\|_{\frac{2(\kappa+1)}{2-\mu}}^{2(\kappa+1)}+Ct^{2(q+\mu-2)+\frac{2(2-\mu)}{\nu}}\|\omega_n\|_{\frac{2q\nu'}{2-\mu}}^{2q}.
\end{align*}
Since $\kappa>2-\mu$, $q>1$, and $\nu$ close to $1$, we have $\Phi_n(t)>0$ for $t>0$ small enough. For $t>0$ large, by \eqref{fff}, we obtain
\begin{align*}
t^{2(\mu-2)}\int_{\mathbb{R}}(I_\mu*F(t\omega_n))F(t\omega_n)\mathrm{d}x \geq
t^{2(\theta+\mu-2)}\int_{\mathbb{R}}(I_\mu*F(\omega_n))F(\omega_n)\mathrm{d}x.\\
\end{align*}
Since $\theta>3-\mu$, we obtain $\Phi_n(t)<0$ for $t>0$ large enough.
Thus $\max\limits_{t\geq0}\Phi_n(t)>0$ is attained at some $t_n>0$.
\end{proof}

\begin{lemma}\label{contr}
Assume that $(f_1)$-$(f_3)$ and $(f_6)$ hold, then there exists $n\in \mathbb{N}^+$ large such that
\begin{align}
\max_{t\geq0}\Phi_n(t)<\frac{2-\mu}{4}.
\end{align}
\end{lemma}

\begin{proof}
First, we give the following estimate in $B_{\frac{1}{n}}(0)$,
\begin{align*}
  \int_{-\frac{1}{n}}^{\frac{1}{n}}\int_{-\frac{1}{n}}^{\frac{1}{n}}\frac{\mathrm{d}x\mathrm{d}y}{|x-y|^\mu}=\frac{2^{3-\mu}}{(1-\mu)(2-\mu)}(\frac{1}{n})^{2-\mu}:=C_\mu (\frac{1}{n})^{2-\mu}.
\end{align*}
By Lemma $\ref{attain}$, we know $\max\limits_{t\geq0}\Phi_n(t)$ is attained at some $t_n>0$. So $t_n$ satisfies
\begin{align*}
\frac{d}{dt}\Phi_n(t)\Big|_{t=t_n}=0.
\end{align*}
By $(f_3)$, we have
\begin{align}\label{twn}
t_n^2\|(-\Delta)^{1/4}\omega_n\|_2^2=&(\mu-2)t_n^{2(\mu-2)}\int_{\mathbb{R}}(I_\mu*F(t_n\omega_n))F(t_n\omega_n)\mathrm{d}x\nonumber\\
&+t_n^{2(\mu-2)}\int_{\mathbb{R}}(I_\mu*F(t_n\omega_n))f(t_n\omega_n)t_n\omega_n\mathrm{d}x\nonumber\\
 \geq& \frac{\theta+\mu-2}{\theta}t_n^{2(\mu-2)}\int_{\mathbb{R}}(I_\mu*F(t_n\omega_n))f(t_n\omega_n)t_n\omega_n\mathrm{d}x.
\end{align}

Note that
\begin{equation*}\label{tF}
  \liminf\limits_{t\rightarrow +\infty}\frac{tF(t)}{e^{\pi t^2}}\geq \liminf\limits_{t\rightarrow +\infty}\frac{\int_0^tsf(s)ds}{e^{\pi t^2}}=\liminf\limits_{t\rightarrow +\infty}\frac{f(t)}{2\pi e^{\pi t^2}}.
\end{equation*}
This with $(f_6)$ yields that, for any $\varepsilon>0$, there exists $R_\varepsilon>0$ such that for any $t\geq R_\varepsilon$,
\begin{equation}\label{ftF}
  f(t)\geq (\beta_0-\varepsilon)e^{\pi t^2},\quad tF(t)\geq \frac{\beta_0-\varepsilon}{2\pi}e^{\pi t^2}.
\end{equation}

$\bf{Case\ 1}$.
 If $\lim\limits_{n\to\infty}t_n^2\log n=0$, then $\lim\limits_{n\to\infty}t_n=0$. By $(\ref{wn})$, we have $\frac{t_n^2}{2}\|(-\Delta)^{1/4}\omega_n\|_2^2\to0$ as $n\rightarrow\infty$. Noted that $F(t_n\omega_n)>0$ by $(f_3)$, so we have
\begin{align*}
0<\Phi_n(t_n)\leq\frac{t_n^2}{2}\|(-\Delta)^{1/4}\omega_n\|_2^2,
\end{align*}
which implies that $\lim\limits_{n\to+\infty}\Phi_n(t_n)=0$, and we conclude.

$\bf{Case\ 2}.
$ If $\lim\limits_{n\to\infty}t_n^2\log n=l\in(0,+\infty]$. From $(\ref{wnx})$, $(\ref{wn})$, $(\ref{twn})$ and $(\ref{ftF})$, we have
\begin{align*}
t_n^2\Big(\frac{\pi a^2 \log n}{4}(1+o(1))\Big)&\geq\frac{\theta+\mu-2}{\theta}t_n^{2(\mu-2)}\int_{B_{\frac{1}{n}}(0)}\int_{B_{\frac{1}{n}}(0)}\frac{F(t_n\omega_n(y))f(t_n\omega_n (x)) t_n\omega_n(x)}{|x-y|^\mu}\mathrm{d}x\mathrm{d}y\\
&\geq\frac{(\theta+\mu-2)(\beta_0-\varepsilon)^2}{2\pi\theta}t_n^{2(\mu-2)}e^{\frac{\pi a^2t_n^2\log ^2n(1+o(1))}{2}}\int_{-\frac{1}{n}}^{\frac{1}{n}}\int_{-\frac{1}{n}}^{\frac{1}{n}}\frac{\mathrm{d}x\mathrm{d}y}{|x-y|^\mu}\\
&=\frac{C_\mu(\theta+\mu-2)(\beta_0-\varepsilon)^2}{2\pi\theta} t_n^{2(\mu-2)}e^{\Big(\frac{\pi a^2t_n^2\log n (1+o(1)))}{2}-(2-\mu)\Big)\log n}.
\end{align*}

(i) If $l=+\infty$, we get a contradiction from the inequality above. So $l\in(0,+\infty)$ and $\lim\limits_{n\rightarrow\infty}t_n=0$. In particular, using the inequality above again and letting $n\to +\infty$, we have $l\in(0,\frac{2(2-\mu)}{\pi a^2}]$.

(ii) If $l\in(0,\frac{2(2-\mu)}{\pi a^2})$, then by $(\ref{wn})$, we get
\begin{align*}
\lim\limits_{n\to\infty}\Phi_n(t_n)\leq \frac{1}{2}\lim\limits_{n\to\infty}t_n^2\|(-\Delta)^{1/4}\omega_n\|_2^2=\frac{\pi a^2 l}{8}<\frac{2-\mu}{4}.
\end{align*}

(iii) If $l=\frac{2(2-\mu)}{\pi a^2}$, by the definition of $\omega_n$, we can find that
\begin{equation*}
  Q_n:=\frac{\pi a^2t_n^2\log n}{2}(1+o(1)))-(2-\mu)\rightarrow 0^+,\quad \text{as $n\rightarrow\infty$}.
\end{equation*}
Using the Taylor's formula, we have
\begin{equation*}
  n^{Q_n}=1+Q_n\log n+\frac{Q_n^2\log ^2 n}{2}+\cdots\geq1.
\end{equation*}
Thus
\begin{equation*}
  \frac{\pi a^2 t_n^2 \log n}{4}=\frac{2-\mu}{2}\geq \frac{C_\mu(\theta+\mu-2)(\beta_0-\varepsilon)^2}{2\pi\theta} t_n^{2(\mu-2)}\rightarrow \infty, \quad \text{as $n\rightarrow\infty$,}
\end{equation*}
which is a contradiction. This ends the proof.
\end{proof}

\section{{\bfseries The monotonicity of the function $a\mapsto m(a)$}}\label{mono}

To guarantee the weak limit of a $(PS)_{m(a)}$ sequence is a ground state solution of problem \ref{problem}, in this section, we investigate the monotonicity of the function $a\mapsto m(a)$.
\begin{lemma}\label{6.1}
Assume that $(f_1)$-$(f_3)$ and $(f_5)$ hold, then the function $a\mapsto m(a)$ is non-increasing on $(0,+\infty)$.
\end{lemma}

\begin{proof}
For any given $a>0$, if $\hat{a}>a$, we prove that $m(\hat{a})\leq m(a)$.
By the definition of $m(a)$, for any $\delta>0$, there exists $u\in\mathcal{P}(a)$ such that
\begin{align}\label{6.1}
J(u)\leq m(a)+\frac{\delta}{3}.
\end{align}
Consider a cut-off function $\varrho\in C_0^{\infty}(\mathbb{R},[0,1])$ such that $\varrho(x)=1$ if $|x|\leq 1$ and $\varrho(x)=0$ if $|x|\geq 2$. For any $\varepsilon>0$ small, define
\begin{align*}
u_{\varepsilon}(x):=\varrho(\varepsilon x)u(x)\in H^{1/2}(\mathbb{R})\backslash\{0\},
\end{align*}
then $u_\varepsilon\to u$ in $H^{1/2}(\mathbb{R})$ as $\varepsilon\to0^+$. From Proposition $\ref{m}$ and Lemma $\ref{pa}$, we have $\beta_{u_\varepsilon}\to \beta_u=0$ in $\mathbb{R}$ and $\mathcal{H}(u_\varepsilon,\beta_{u_\varepsilon})\to \mathcal{H}(u,\beta_u)=u$ in $H^{1/2}(\mathbb{R})$ as $\varepsilon\to0^+$.
Fix $\varepsilon_0>0$ small enough such that
\begin{align}\label{6.2}
 J(\mathcal{H}(u_{\varepsilon_0}, \beta_{u_{\varepsilon_0}}))\leq J(u)+\frac{\delta}{3}.
\end{align}
Let $v\in C_0^{\infty}(\mathbb{R})$ satisfy $supp (v)\subset B_{1+\frac{4}{\varepsilon_0}}(0)\backslash B_{\frac{4}{\varepsilon_0}}(0)$, and set
\begin{align*}
v_{\varepsilon_0}=\frac{\hat{a}^2-\|u_{\varepsilon_0}\|_2^2}{\|v\|_2^2}v.
\end{align*}
Define $s_h:=u_{\varepsilon_0}+\mathcal{H}(v_{\varepsilon_0},h)$ for $h<0$. Since $dist(u_{\varepsilon_0},\mathcal{H}(v_{\varepsilon_0},h))\geq \frac{2}{\varepsilon_0}>0$, we obtain $\|s_h\|_2^2=\hat{a}^2$, i.e., $s_h\in S(\hat{a})$.
We claim that  $\beta_{s_h}$ is bounded from above as $h\to-\infty$. Otherwise, by $(f_3)$, $(\ref{fff})$ and $s_h\to u_{\varepsilon_0}\neq 0$ a.e. in $\mathbb{R}$ as $h\to-\infty$, one has
\begin{align*}
0\leq \lim_{n\to\infty}e^{-\beta_{s_h}}J(\mathcal{H}(s_h, \beta_{s_h}))
&\leq \lim_{n\to\infty} \frac{1}{2}\Big[\|(-\Delta)^{1/4}s_h\|_2^2-e^{(\theta+\mu-3)\beta_{s_h}}\int_{\mathbb{R}}(I_\mu*F(s_h))F(s_h)\mathrm{d}x
\Big]\\
&=-\infty,
\end{align*}
which leads to a contradiction.
Thus $\beta_{s_h}+h\to-\infty$ as $h\to-\infty$, by $(f_3)$, we get
\begin{align}\label{6.3}
J(\mathcal{H}(v_{\varepsilon_0}, \beta_{s_h}+h))\leq \frac{e^{\beta_{s_h}+h}}{2}\|(-\Delta)^{1/4}v_{\varepsilon_0}\|_2^2\to0,\quad \mbox{as}\ h\to-\infty.
\end{align}
We deduce from Lemma $\ref{pa}$ and $(\ref{6.1})$-$(\ref{6.3})$ that
\begin{align*}
m(\hat{a})\leq J(\mathcal{H}(s_h, \beta_{s_h}))
=&J(\mathcal{H}(u_{\varepsilon_0}, \beta_{s_h}))+J(\mathcal{H}(\mathcal{H}(v_{\varepsilon_0},h), \beta_{s_h}))\\
=& J(\mathcal{H}(u_{\varepsilon_0}, \beta_{s_h}))+J(\mathcal{H}(v_{\varepsilon_0}, \beta_{s_h}+h))\\
\leq& J(\mathcal{H}(u_{\varepsilon_0}, \beta_{u_{\varepsilon_0}}))+J(\mathcal{H}(v_{\varepsilon_0}, \beta_{s_h}+h))\\
\leq& m(a)+\delta.
\end{align*}
By the arbitrariness of $\delta>0$, we deduce that $m(\hat{a})\leq m(a)$ for any $\hat{a}>a$.
\end{proof}

\begin{lemma}\label{6.2}
Assume that $(f_1)$-$(f_3)$ and $(f_5)$ hold. Suppose that problem $(\ref{problem})$ possesses a ground state solution $u$ with $\lambda<0$, then $m(a^*)< m(a)$ for any $a^*>a$ close to $a$.
\end{lemma}

\begin{proof}
For any $t>0$ and $\beta\in\mathbb{R}$, one has $\mathcal{H}(tu,\beta)\in S(ta)$ and
\begin{align*}
J(\mathcal{H}(tu,\beta))=\frac{t^2e^\beta}{2}\|(-\Delta)^{1/4}u\|_2^2-\frac{e^{(\mu-2)\beta}}{2}
\int_{\mathbb{R}}(I_\mu*F(te^\frac{\beta}{2}u))F(te^\frac{\beta}{2}u)\mathrm{d}x.
\end{align*}
Then
\begin{align*}
\frac{\partial{J(\mathcal{H}(tu,\beta))}}{\partial t}&=te^\beta \|(-\Delta)^{1/4}u\|_2^2-e^{(\mu-2)\beta}
\int_{\mathbb{R}}(I_\mu*F(te^\frac{\beta}{2}u))f(te^\frac{\beta}{2}u)e^{\frac{\beta}{2}}u\mathrm{d}x\\&=\frac{\langle J'(\mathcal{H}(tu,\beta)),\mathcal{H}(tu,\beta) \rangle}{t}.
\end{align*}
For convenience, we denote $\tau(t,\beta):=J(\mathcal{H}(tu,\beta))$.
By Proposition $\ref{m}$, $\mathcal{H}(tu,\beta)\to u$ in $H^{1/2}(\mathbb{R})$ as $(t,\beta)\to(1,0)$.
Since $\lambda<0$, we have $\langle J'(u),u \rangle=\lambda\|u\|_2^2=\lambda a^2<0$.
Hence, for $\delta>0$ small enough, one has
\begin{align*}
\frac{\partial{\tau(t,\beta)}}{\partial t}<0 \quad\mbox{for any}\ (t,\beta)\in (1,1+\delta]\times[-\delta,\delta].
\end{align*}
For any $t\in(1,1+\delta]$ and $\beta\in [-\delta,\delta]$, using the mean value theorem, we obtain
\begin{align*}
\tau(t,\beta)=\tau(1,\beta)+(t-1)\cdot\frac{\partial{\tau(t,\beta)}}{\partial t}\Big|_{t=\xi}<\tau(1,\beta).
\end{align*}
for some $\xi\in(1,t)$.
By Lemma $\ref{pa}$, $\beta_{tu}\to\beta_u=0$ in $\mathbb{R}$ as $t\to 1^+$. For any $a^*>a$ close to $a$, let $\hat t=\frac{a^*}{a}$, then
$\hat t\in (1,1+\delta]$ and $\beta_{\hat t u}\in[-\delta,\delta]$.
Applying Lemma $\ref{pa}$ again, we have
\begin{align*}
m(a^*)\leq \tau(\hat t,\beta_{\hat t u})<\tau(1,\beta_{\hat tu})=J(\mathcal{H}(u, \beta_{\hat t u}))\leq J(u)=m(a).
\end{align*}
\end{proof}
From Lemmas \ref{6.1} and \ref{6.2}, we immediately have the following result.
\begin{lemma}\label{6.3}
Assume that $(f_1)$-$(f_3)$ and $(f_5)$ hold. Suppose that problem $(\ref{problem})$ possesses a ground state solution $u$ with $\lambda<0$, then $a\mapsto m(a)$ is decreasing on $(0,+\infty)$. 
\end{lemma}

\section{{\bfseries Palais-Smale sequence}}\label{ps}
In this section, using the minimax principle based on the homotopy stable family of compact subsets of $S(a)$ (see \cite{ghou} for more details), we construct a $(PS)_{m(a)}$ sequence on $\mathcal{P}(a)$ for $J|_{S(a)}$.

\begin{proposition}\label{pro}
Assume that $(f_1)$-$(f_3)$ and $(f_5)$ hold, then there exists a $(PS)_{m(a)}$ sequence $\{u_n\}\subset\mathcal{P}(a)$ for $J|_{S(a)}$.
\end{proposition}

Following by \cite{willem}, we recall that for any $a>0$, the tangent space of $S(a)$ at $u$ is defined by
\begin{align*}
T_u:=\Big\{u\in H^{1/2}(\mathbb{R}) : \int_{\mathbb{R}}uv\mathrm{d}x=0\Big\}.
\end{align*}
To prove Proposition $\ref{pro}$, we borrow some arguments from \cite{soave2} and consider the functional $\mathcal{I}: S(a)\to\mathbb{R}$ defined by
\begin{align*}
\mathcal{I}(u)=J(\mathcal{H}(u,\beta_u)),
\end{align*}
where $\beta_u\in\mathbb{R}$ is the unique number obtained in Lemma \ref{pa} for any $u\in S(a)$. By Lemma \ref{pa}, we know that $\beta_u$ is continuous as a mapping for any $u\in S(a)$. However, it remains unknown that whether $\beta_u$ is of class $C^1$. Inspired by \cite[Proposition 2.9]{sz}, we have
\begin{lemma}\label{c1}
Assume that $(f_1)-(f_3)$ and $(f_5)$ hold, then the functional $\mathcal{I}: S(a)\to \mathbb{R}$ is of class $C^1$ and
\begin{align*}
\langle \mathcal{I}^{\prime}(u), v\rangle=&\frac{e^{\beta_u}}{2\pi}\int_{\mathbb{R}}\int_{\mathbb{R}}\frac{|u(x)-u(y)||v(x)-v(y)|}{|x-y|^2}\mathrm{d}x\mathrm{d}y
-e^{(\mu-2)\beta_u}\int_{\mathbb{R}}(I_\mu*F(e^{\frac{\beta_u}{2}}u))f(e^{\frac{\beta_u}{2}}u)e^{\frac{\beta_u}{2}}v\mathrm{d}x\\
=&\langle J^{\prime}(\mathcal{H}(u,\beta_u)) ,\mathcal{H}(v,\beta_u)\rangle
\end{align*}
for any $u\in S(a)$ and $v\in T_u$.
\end{lemma}
\begin{proof}
Let $u\in S(a)$ and $v\in T_u$, for any $|t|$ small enough, by Lemma $\ref{pa}$,
\begin{align*}
\mathcal{I}(u+tv)-\mathcal{I}(u)=&J\big(\mathcal{H}(u+tv,\beta_{u+tv})\big)-J\big(\mathcal{H}(u,\beta_u)\big)\\
\leq &J\big(\mathcal{H}(u+tv,\beta_{u+tv})\big)-J\big(\mathcal{H}(u,\beta_{u+tv})\big)\\
=&\frac{1}{2}e^{\beta_{u+tv}}\Big[\|(-\Delta)^{1/4}(u+tv)\|_2^2-\|(-\Delta)^{1/4}u\|_2^2
\Big]\\
&-\frac{1}{2}e^{(\mu-2)\beta_{u+tv}}\int_{\mathbb{R}}\Big[\big(I_\mu*F(e^{\frac{\beta_{u+tv}}{2}}(u+tv))\big)F(e^{\frac{\beta_{u+tv}}{2}}(u+tv))\\
&-(I_\mu*F(e^{\frac{\beta_{u+tv}}{2}}u))F(e^{\frac{\beta_{u+tv}}{2}}u)
\Big]\mathrm{d}x\\
=&\frac{1}{2}e^{\beta_{u+tv}}\Big[t^2\|(-\Delta)^{1/4}v\|_2^2+2t\frac{1}{2\pi}\int_{\mathbb{R}}\int_{\mathbb{R}}\frac{|u(x)-u(y)||v(x)-v(y)|}{|x-y|^2}\mathrm{d}x\mathrm{d}y\Big]\\
&-\frac{1}{2}e^{(\mu-2)\beta_{u+tv}}\int_{\mathbb{R}}\big(I_\mu*F(e^{\frac{\beta_{u+tv}}{2}}(u+tv))\big)f(e^{\frac{\beta_{u+tv}}{2}}(u+\xi_ttv))e^{\frac{\beta_{u+tv}}{2}}tv\mathrm{d}x\\
&-\frac{1}{2}e^{(\mu-2)\beta_{u+tv}}\int_{\mathbb{R}}\big(I_\mu*F(e^{\frac{\beta_{u+tv}}{2}}u)\big)f(e^{\frac{\beta_{u+tv}}{2}}(u+\xi_ttv))e^{\frac{\beta_{u+tv}}{2}}tv\mathrm{d}x,
\end{align*}
where $\xi_t\in(0,1)$. On the other hand,
\begin{align*}
\mathcal{I}(u+tv)-\mathcal{I}(u)=&J\big(\mathcal{H}(u+tv,\beta_{u+tv})\big)-J\big(\mathcal{H}(u,\beta_u)\big)\\
\geq &J\big(\mathcal{H}(u+tv,\beta_u)\big)-J\big(\mathcal{H}(u,\beta_{u})\big)\\
\geq& \frac{1}{2}e^{\beta_{u}}\Big[t^2\|(-\Delta)^{1/4}v\|_2^2+2t\frac{1}{2\pi}\int_{\mathbb{R}}\int_{\mathbb{R}}\frac{|u(x)-u(y)||v(x)-v(y)|}{|x-y|^2}\mathrm{d}x\mathrm{d}y\Big]\\
&-\frac{1}{2}e^{(\mu-2)\beta_{u}}\int_{\mathbb{R}}\big(I_\mu*F(e^{\frac{\beta_{u}}{2}}(u+tv))\big)f(e^{\frac{\beta_{u}}{2}}(u+\zeta_ttv))e^{\frac{\beta_{u}}{2}}tv\mathrm{d}x\\
&-\frac{1}{2}e^{(\mu-2)\beta_{u}}\int_{\mathbb{R}}\big(I_\mu*F(e^{\frac{\beta_{u}}{2}}u)\big)f(e^{\frac{\beta_{u}}{2}}(u+\zeta_ttv))e^{\frac{\beta_{u}}{2}}tv\mathrm{d}x,
\end{align*}
where $\zeta_t\in(0,1)$. By Lemma $\ref{pa}$, $\lim\limits_{t\to0}\beta_{u+tv}=\beta_u$, from the above inequalities, we conclude
\begin{align*}
\lim_{t\to0}\frac{\mathcal{I}(u+tv)-\mathcal{I}(u)}{t}=&\frac{e^{\beta_u}}{2\pi}\int_{\mathbb{R}}\int_{\mathbb{R}}\frac{|u(x)-u(y)||v(x)-v(y)|}{|x-y|^2}\mathrm{d}x\mathrm{d}y\\
&-e^{(\mu-2)\beta_{u}}\int_{\mathbb{R}}\big(I_\mu*F(e^{\frac{\beta_{u}}{2}}u)\big)f(e^{\frac{\beta_{u}}{2}}u)e^{\frac{\beta_{u}}{2}}v\mathrm{d}x.
\end{align*}
Using Lemma $\ref{pa}$ again, We find that the G\^ateaux derivative of $\mathcal{I}$ is continuous linear in $v$ and continuous in $u$.
Therefore, by \cite[Proposition 1.3]{willem}, we obtain $\mathcal{I}$ is of class $C^1$. Changing variables in the integrals, we can prove the rest.
\end{proof}

\begin{lemma}\label{Minimax}
Assume that $(f_1)$-$(f_3)$ and $(f_5)$ hold. Let $\mathcal{F}$ be a homotopy stable family of compact subsets of $S(a)$ without boundary and set
\begin{align*}
m_{\mathcal{F}}:=\inf_{A\in\mathcal{F}}\max_{u\in A}\mathcal{I}(u).
\end{align*}
If $m_\mathcal{F}>0$, then there exists a $(PS)_{m_{\mathcal{F}}}$ sequence $\{u_n\}\subset \mathcal{P}(a)$ for $J|_{S(a)}$.
\end{lemma}
\begin{proof}
Let $\{A_n\}\subset \mathcal{F}$ be a minimizing sequence of $m_{\mathcal{F}}$. We define the mapping $\eta : [0,1]\times S(a)\to S(a)$, that is $\eta(t,u)=\mathcal{H}(u, t\beta_u)$. By Proposition \ref{m} and Lemma $\ref{pa}$, $\eta(t, u)$ is continuous in $[0,1]\times S(a)$ and satisfies $\eta(t,u)=u$ for all $(t,u)\in \{0\}\times S(a)$. Thus by the definition of $\mathcal{F}$ (see \cite[Definition 3.1]{ghou}), one has
\begin{align*}
Q_n:=\eta(1, A_n)=\{\mathcal{H}(u,\beta_u) : u\in A_n\}\subset \mathcal{F}.
\end{align*}
Obviously, $Q_n\subset \mathcal{P}(a)$ for any $n\in\mathbb{N}^+$. Since $\mathcal{I}(\mathcal{H}(u,\beta))=\mathcal{I}(u)$ for any $u\in S(a)$ and $\beta\in\mathbb{R}$, then
\begin{align*}
\max_{u\in Q_n}\mathcal{I}(u)=\max_{u\in A_n}\mathcal{I}(u)\to m_{\mathcal{F}},\quad \mbox{as $n\rightarrow\infty$},
\end{align*}
which implies that $\{Q_n\}\subset \mathcal{F}$ is another minimizing sequence of $m_{\mathcal{F}}$. Since $G(u):=\|u\|_2^2-a^2$ is of class $C^1$, and for any $u\in S(a)$, we have $\langle G'(u),u\rangle=2a^2>0$. Therefore, by the implicit function theorem, $S(a)$ is a $C^1$-Finsler manifold.  By \cite[Theorem 3.2]{ghou}, we obtain a $(PS)_{m_{\mathcal{F}}}$ sequence $\{v_n\}\subset S(a)$ for $\mathcal{I}$ such that $\lim\limits_{n\to+\infty}dist(v_n, Q_n)=0$.
Let
\begin{align*}
u_n:=\mathcal{H}(v_n,\beta_{v_n}),
\end{align*}
we prove that $\{u_n\}\subset \mathcal{P}(a)$ is the desired sequence.
We claim that there exists $C>0$ such that $e^{-\beta_{v_n}}\leq C$ for any $n\in\mathbb{N}^+$. Indeed, we have
\begin{align*}
e^{-\beta_{v_n}}=\frac{\|(-\Delta)^{1/4}v_n\|_2^2}{\|(-\Delta)^{1/4}u_n\|_2^2}.
\end{align*}
Since $\{u_n\}\subset \mathcal{P}(a)$, by Lemma $\ref{ous1}$, we know that there exists a constant $C>0$ such that $\|(-\Delta)^{1/4}u_n\|_2^2\geq C$ for any $n\in\mathbb{N}^+$. Since $Q_n\subset\mathcal{P}(a)$ for any $n\in\mathbb{N}^+$ and for any $u\in\mathcal{P}(a)$, one has $J(u)=\mathcal{I}(u)$, then
\begin{align*}
\max_{u\in Q_n}J(u)=\max_{u\in Q_n}\mathcal{I}(u)\to m_{\mathcal{F}}, \quad \mbox{as} \ n\to+\infty.
\end{align*}
This fact together with $Q_n\subset \mathcal{P}(a)$ and $(f_3)$ yields that $\{Q_n\}$ is uniformly bounded in $H^{1/2}(\mathbb{R})$, 
thus from $\lim\limits_{n\to \infty}dist(v_n,Q_n)=0$, we obtain $\sup\limits_{n\geq 1}\|v_n\|_{1/2}^2<+\infty$. This prove the claim.

Since $\{u_n\}\subset \mathcal{P}(a)$, one has $J(u_n)=\mathcal{I}(u_n)=\mathcal{I}(v_n)\to m_{\mathcal{F}}$ as $n\to\infty$. For any $\phi\in T_{u_n}$, we have
\begin{align*}
\int_{\mathbb{R}}v_n\mathcal{H}(\phi,-\beta_{v_n})\mathrm{d}x=\int_{\mathbb{R}}\mathcal{H}(v_n,\beta_{v_n})\phi\mathrm{d}x=\int_{\mathbb{R}}u_n\phi\mathrm{d}x=0,
\end{align*}
which implies that $\mathcal{H}(\phi,-\beta_{v_n})\in T_{v_n}$. Also,
\begin{align*}
\|\mathcal{H}(\phi,-\beta_{v_n})\|_{1/2}^2=e^{-\beta_{v_n}}\|(-\Delta)^{1/4}\phi\|_2^2+\|\phi\|_2^2
\leq C\|(-\Delta)^{1/4}\phi\|_2^2+\|\phi\|_2^2
\leq \max\{1,C\}\|\phi\|_{1/2}^2.
\end{align*}
By Lemma $\ref{c1}$, for any $\phi\in T_{u_n}$, we deduce that
\begin{align*}
\big|\langle J^{\prime}(u_n),\phi\rangle\big|=&\Big|\langle J^{\prime}\big(\mathcal{H}(v_n,\beta_{v_n})\big), \mathcal{H}\big(\mathcal{H}(\phi,-\beta_{v_n}),\beta_{v_n}\big)\rangle\Big|
=\Big|\big\langle \mathcal{I}^{\prime}(v_n),\mathcal{H}(\phi,-\beta_{v_n})\big\rangle\Big|\\
\leq&\|\mathcal{I}^{\prime}(v_n)\|_*\cdot\|\mathcal{H}(\phi,-\beta_{v_n})\|_{1/2}
\leq\max\{1, \sqrt{C}\}\|\mathcal{I}^{\prime}(v_n)\|_*\cdot\|\phi\|_{1/2},
\end{align*}
where $(Y^*,\|\cdot\|_*)$ is the dual space of Banach space $(Y,\|\cdot\|_{1/2})$. Hence we can deduce that
\begin{align*}
\|J^{\prime}(u_n)\|_* \leq\max\big\{1, \sqrt{C}\big\}\|\mathcal{I}^{\prime}(v_n)\|_*\to0, \quad \mbox{as}\ n\to\infty,
\end{align*}
 which implies that  $\{u_n\}$ is a $(PS)_{m_{\mathcal{F}}}$ sequence for $J|_{S(a)}$. This ends the proof.
\end{proof}

\noindent{\bfseries Proof of Proposition \ref{pro}.}
Note that the class $\mathcal{F}$ of all singletons included in $S(a)$ is a homotopy stable family of compact subsets of $S(a)$ without boundary.
By Lemma $\ref{Minimax}$, we know that if $m_{\mathcal{F}}>0$, then there exists a $(PS)_{m_{\mathcal{F}}}$ sequence $\{u_n\}\subset \mathcal{P}(a)$ for $J|_{S(a)}$.
By Lemma \ref{ous1}, we know $m(a)>0$, so if we can prove that $m_{\mathcal{F}}=m(a)$, then we complete the proof.

In fact,
by the definition of $\mathcal{F}$, we have
\begin{align*}
m_{\mathcal{F}}=\inf_{A\in\mathcal{F}}\max_{u\in A}\mathcal{I}(u)=\inf_{u\in S(a)}\mathcal{I}(u)=\inf_{u\in S(a)}\mathcal{I}(\mathcal{H}(u,\beta_u))=\inf_{u\in S(a)}J(\mathcal{H}(u,\beta_u)).
\end{align*}
For any $u\in S(a)$, it follows from $\mathcal{H}(u,\beta_u)\in \mathcal{P}(a)$ that $J(\mathcal{H}(u,\beta_u))\geq m(a)$, so $m_{\mathcal{F}}\geq m(a)$.
On the other hand, for any $u\in \mathcal{P}(a)$, by Lemma $\ref{pa}$, we deduce that $\beta_u=0$ and $J(u)=J(\mathcal{H}(u,0))\geq \inf\limits_{u\in S(a)}J(\mathcal{H}(u,\beta_u))$, which implies that $m(a)\geq m_{\mathcal{F}}$.
  \qed

For the sequence $\{u_n\}$ obtained in Proposition $\ref{pro}$, by $(f_3)$, we know that $\{u_n\}$ is bounded in $H^{1/2}(\mathbb{R})$. Up to a subsequence, we assume that $u_n\rightharpoonup u$ in $H^{1/2}(\mathbb{R})$. Furthermore, by $J\big|^{\prime}_{S(a)}(u_n)\to0$ as $n\to+\infty$ and the Lagrange multiplier rule, there exists $\{\lambda_n\}\subset \mathbb{R}$ such that
\begin{align}\label{key}
(-\Delta)^{1/2}u_n=\lambda_nu_n+(I_\mu*F(u_n))f(u_n)+o_n(1).
\end{align}

\begin{lemma}\label{non}
Assume that $(f_1)$-$(f_3)$ and $(f_5)$-$(f_6)$ hold, then up to a subsequence and up to translations in $\mathbb{R}$, $u_n\rightharpoonup u\neq0$ in $H^{1/2}(\mathbb{R})$.
\end{lemma}

\begin{proof}
We claim that
\begin{align*}
\Lambda:=\limsup_{n\to+\infty}\Big(\sup_{y\in\mathbb{R}}\int_{B_r(y)}|u_n|^2\mathrm{d}x\Big)>0.
\end{align*}
If this is false, we obtain $u_n\to0$ in $L^p(\mathbb{R})$ for any $p>2$ by the Lions' vanishing lemma \cite[Lemma 1.21]{willem}.
From $J(u_n)=m(a)+o_n(1)$, $P(u_n)=0$ and $(f_3)$, we have
\begin{align*}
J(u_n)-\frac{1}{2}P(u_n)\geq \frac{\theta+\mu-3}{2 \theta}\int_{\mathbb{R}}(I_\mu*F(u_n))f(u_n)u_n\mathrm{d}x.
\end{align*}
Since $\theta>3-\mu$, we get
\begin{align}\label{1n}
\limsup_{n\to\infty}\int_{\mathbb{R}}(I_\mu*F(u_n))f(u_n)u_n\mathrm{d}x\leq \frac{2\theta m(a)}{\theta+\mu-3}=:K_0.
\end{align}
For any $\delta\in (0,\frac{M_0K_0}{t_0})$, from $(f_3)$ and $(f_4)$, we can choose $M_\delta>\frac{M_0K_0}{\delta}>t_0$, then
\begin{align}\label{l}
\int_{|u_n|\geq M_\delta}(I_\mu*F(u_n))F(u_n)\mathrm{d}x\leq& M_0\int_{|u_n|\geq M_\delta}(I_\mu*F(u_n))|f(u_n)|\mathrm{d}x\nonumber\\
\leq& \frac{M_0}{M_\delta}\int_{|u_n|\geq M_\delta}(I_\mu*F(u_n))f(u_n)u_n\mathrm{d}x< \delta.
\end{align}
On the other hand, by \eqref{HLS} and $(\ref{Ft})$, we have
\begin{align}\label{ll}
\int_{|u_n|\leq M_\delta}(I_\mu*F(u_n))F(u_n)\mathrm{d}x\leq C(\|u_n\|_{\kappa+1}^{\kappa+1}+\|u_n\|_q^q)=o_n(1).
\end{align}
Combining $(\ref{l})$, $(\ref{ll})$ with the arbitrariness of $\delta$, we obtain
\begin{align*}
\int_{\mathbb{R}}(I_\mu*F(u_n))F(u_n)\mathrm{d}x=o_n(1).
\end{align*}
Thus, by Lemma \ref{contr}, we have
\begin{align*}
\limsup_{n\to\infty}\|(-\Delta)^{1/4}u_n\|_2^2\leq 2m(a)<\frac{2-\mu}{2}.
\end{align*}
Up to a subsequence, we assume that $\sup\limits_{n\in\mathbb{N}^+}\|(-\Delta)^{1/4}u_n\|_2^2<\frac{2-\mu}{2}.$ Fix $\alpha>\pi$ close to $\pi$ and $\nu>1$ close to $1$ such that
\begin{align*}
\sup_{n\in\mathbb{N}^+}\frac{2\alpha \nu \| (-\Delta)^{1/4}u_n\|_2^2}{2-\mu}< \pi.
\end{align*}
From $(\ref{tain1})$,  for $\nu'=\frac{\nu}{\nu-1}$, we have
\begin{align*}
\|F(u_n)\|_{\frac{2}{2-\mu}}\leq
C\|u_n\|_{\frac{2(\kappa+1)}{2-\mu}}^{\kappa+1}+C
\|u_n\|_{\frac{2q\nu'}{2-\mu}}^q\to0,\quad \mbox{as}\ n\to+\infty.
\end{align*}
By a similar argument as above, we infer that $ \|f(u_n)u_n\|_{\frac{2}{2-\mu}}\to0$ as $n\to\infty$.
Hence, we obtain
\begin{align*}
\int_{\mathbb{R}}(I_\mu*F(u_n))f(u_n)u_n\mathrm{d}x=o_n(1).
\end{align*}
Since $P(u_n)=0$, we have $\| (-\Delta)^{1/4}u_n\|_2^2=o_n(1)$, then $m(a)=0$, which is a contradiction. According to $\Lambda>0$, there exists $\{y_n\}\subset\mathbb{R}$ such that $\int_{B_1(y_n)}|u_n|^2\mathrm{d}x>\frac{\Lambda}{2}$, i.e., $\int_{B_1(0)}|u_n(x-y_n)|^2\mathrm{d}x>\frac{\Lambda}{2}$. Then up to a subsequence and up to translations in $\mathbb{R}$, $u_n\rightharpoonup u\neq0$ in $H^{1/2}(\mathbb{R})$.
\end{proof}

\begin{lemma}\label{ne}
Assume that $(f_1)$-$(f_3)$ and $(f_5)$-$(f_6)$ hold. Then $\{\lambda_n\}$ is bounded in $\mathbb{R}$ and $\lambda_n\to\lambda$ with some $\lambda<0$ as $n\to\infty$.
\end{lemma}

\begin{proof}
According to $(\ref{key})$, we have
\begin{align*}
\|(-\Delta)^{1/4}u_n\|_2^2=\lambda_n\int_{\mathbb{R}}|u_n|^2\mathrm{d}x+
\int_{\mathbb{R}}(I_\mu*F(u_n))f(u_n)u_n\mathrm{d}x+o_n(1).
\end{align*}
Combining with $P(u_n)=0$, we get
\begin{align*}
\lambda_na^2=-(2-\mu)
\int_{\mathbb{R}}(I_\mu*F(u_n))F(u_n)\mathrm{d}x+o_n(1).
\end{align*}
Thus
\begin{align}\label{nn}
\limsup_{n\to\infty}|\lambda_n|\leq \frac{2-\mu}{a^2}
\int_{\mathbb{R}}(I_\mu*F(u_n))F(u_n)\mathrm{d}x,
\end{align}
this together with $(f_3)$ and $(\ref{1n})$ yields that $\{\lambda_n\}$ is bounded in $\mathbb{R}$.
 Moreover, by Lemma \ref{non} and Fatou Lemma, we have
\begin{align*}
\limsup_{n\to\infty}\lambda_n=-\liminf_{n\to\infty}\frac{2-\mu}{a^2}
\int_{\mathbb{R}}(I_\mu*F(u_n))F(u_n)\mathrm{d}x\leq -\frac{2-\mu}{a^2}
\int_{\mathbb{R}}(I_\mu*F(u))F(u)\mathrm{d}x<0.
\end{align*}
Therefore, up to a subsequence, $\lambda_n\to\lambda$ with some $\lambda<0$ as $n\to\infty$.
\end{proof}

\section{{\bfseries Proof of the result}}\label{proof}

\noindent{\bfseries Proof of Theorem \ref{thm1.1}.} Under the assumptions of Theorem \ref{thm1.1}, from \eqref{key}, \eqref{1n}, Lemmas \ref{f}, \ref{non}, \ref{ne},
we know $u$ is a weak solution of problem $(\ref{problem})$ with $\lambda<0$ and $P(u)=0$.
Using the Br\'ezis-Lieb Lemma\cite[Lemma 1.32]{willem}, we get
\begin{align*}
\int_{\mathbb{R}}|u_n|^2\mathrm{d}x=\int_{\mathbb{R}}|u_n-u|^2\mathrm{d}x+\int_{\mathbb{R}}|u|^2\mathrm{d}x+o_n(1).
\end{align*}
Let $a_1:=\|u\|_2>0$ and $a_{n,2}:=\|u_n-u\|_2$, then $a^2=a_1^2+a_{n,2}^2+o_n(1)$.
On the one hand, using $(f_3)$, $P(u)=0$ and Fatou Lemma, we have
\begin{align*}
J(u)=&J(u)-\frac{1}{2}P(u)=\frac{1}{2}\int_{\mathbb{R}}\Big[(I_\mu*F(u))f(u)u-(3-\mu)(I_\mu*F(u))F(u)\Big]\mathrm{d}x\\
\leq& \liminf_{n\to\infty} \frac{1}{2}\int_{\mathbb{R}}\Big[(I_\mu*F(u))f(u)u-(3-\mu)(I_\mu*F(u))F(u)\Big]\mathrm{d}x\\
=&\liminf_{n\to\infty}(J(u_n)-\frac{1}{2}P(u_n))
=m(a).
\end{align*}
On the other hand, it follows from Lemma $\ref{6.1}$ that $J(u)\geq m(a_1)\geq m(a).$ Thus $J(u)= m(a_1)= m(a)$. 
By Lemma \ref{6.3}, we obtain $a=a_1$.
This implies $u$ is a ground state solution of $(\ref{problem})$.
\qed

\end{document}